\documentclass[aap,preprint]{imsart}

\RequirePackage{amsthm,amsmath,amsfonts,amssymb}
\RequirePackage[numbers]{natbib}
\RequirePackage[colorlinks,citecolor=blue,urlcolor=blue]{hyperref}
\RequirePackage{graphicx}

\startlocaldefs

\usepackage[utf8]{inputenc}
\usepackage{float}
\usepackage[francais,english]{babel}
\usepackage{enumerate}
\usepackage{enumitem}

\usepackage{url}
\usepackage{comment}

\usepackage{amsmath,xcolor,color,dsfont,verbatim}
\usepackage{amsfonts}
\usepackage{amssymb,amsthm}

\usepackage{mathrsfs}

\usepackage{graphicx}

\usepackage{subfigure}

\usepackage{lscape}

\colorlet{darkgreen}{green!50!black}

\usepackage{hyperref} 
\hypersetup{	
pdftitle={Asymptotic behavior of the occupancy density for RBM in a half-plane and Martin boundary}, 
pdfauthor={Philip Ernst and Sandro Franceschi},
}

\newtheorem{thm}{Theorem}

\newtheorem{prop}[thm]{Proposition}
\newtheorem{cor}[thm]{Corollary}
\newtheorem{deff}[thm]{Definition}

\newtheorem{lem}[thm]{Lemma}

\theoremstyle{definition}

\newtheorem*{remarque*}{Remarque}
\newtheorem*{remark*}{Remark}
\newtheorem{remark}[thm]{Remark}

\renewcommand{\geq}{\geqslant}
\renewcommand{\leq}{\leqslant}

\usepackage{dsfont}

\usepackage{enumitem}

\usepackage{mathtools}
\usepackage{eqnarray,cases}
\usepackage{multirow}

\usepackage{indentfirst}

\usepackage{pstricks-add}
\usepackage[normalem]{ulem}

\usepackage{float}

\let\phi=\varphi
\let\epsilon=\varepsilon

\usepackage{abstract}

\usepackage{bm}

\newcommand{\rems}[1]{\textcolor{black}{#1}}

\newcommand{\pii}{\rems{\bm{\pi}}}

\endlocaldefs

\begin{document}

\begin{frontmatter}

\title{Asymptotic behavior of the occupancy density for obliquely reflected Brownian motion in a half-plane and Martin boundary}
\runtitle{Asymptotic behavior of the occupancy density for RBM in a half-plane}

\begin{aug}

  \author[A]{\fnms{Philip A.} \snm{Ernst}\ead[label=e1]{philip.ernst@rice.edu}} \and
  \author[B]{\fnms{Sandro} \snm{Franceschi}\ead[label=e2]{sandro.franceschi@universite-paris-saclay.fr}}

\address[A]{Department of Statistics,
Rice University
}

\address[B]{Laboratoire de Math\'ematiques d’Orsay,
Universit\'e Paris-Saclay
}

\end{aug}

\renewcommand{\abstractname}{}

\begin{abstract} \
Let $\pii$ be the occupancy density 
of an obliquely reflected Brownian motion in the half plane and let $(\rho,\alpha)$ be the polar coordinates of a point in the upper half plane. This work determines the exact asymptotic behavior of $\pii (\rho,\alpha)$ as $\rho\to \infty$ with $\alpha\in(0,\pi)$. We find explicit functions $a,b,c$ such that
$$
\pii (\rho,\alpha) \underset{\rho\to\infty}{\sim} a(\alpha) \rho^{b(\alpha)} e^{-c(\alpha)\rho}.
$$
This closes an open problem first stated by Professor J. Michael Harrison in August 2013.
We also compute the exact asymptotics for the tail distribution of the boundary occupancy measure and we obtain an explicit integral expression for $\pii$. {We conclude by finding the Martin boundary of the process and giving} all of the corresponding harmonic functions satisfying an oblique Neumann boundary problem.
\end{abstract}


\end{frontmatter}

\tableofcontents

\section{Introduction}\label{sec:intro}

\noindent In 2013, Professor J. Michael Harrison raised a fundamental question regarding the asymptotic behavior of the occupancy density for reflected Brownian motion (RBM) in the half plane \cite{Web}. We shall state Harrison's problem on the following page after introducing the necessary background for the statement of the problem. The purpose of the present paper is to close this open problem. \\
\indent  Let 
$B(t)+\mu t $ be a two-dimensional Brownian motion with identity covariance matrix, drift vector $\mu=(\mu_1,\mu_2)$, and initial state $(0,0)$.\footnote{Appendix~\ref{appendix:generalization} generalize our results to any covariance matrix and to any starting point.} Let $R=(r,1)$ be \rems{a} reflection vector and, for all $t\geqslant 0$, let
$$
\ell(t):= - \underset{0 \leqslant s \leqslant t}{\inf} {(B_2 (s) +\mu_2 s)} \,\,\,\, \text{ and }\,\,\,\,
Z(t):=B(t)+\mu t+ R \ell(t) \in \mathbb{R}\times \mathbb{R}_+.
$$
It is said that $(Z, \ell)$ solves the Skorokhod problem for $B(t)+\mu t$ with respect to upper half-plane and to $R$. 
The process $Z$ is a reflected Brownian motion (RBM) in the upper half-plane and $\ell$ is the  local time of $Z$ on the abscissa.
We shall 
assume throughout that
\begin{equation}
\mu_1+r\mu_2^- <0,
\label{eq:conditiontrans}
\footnote{The symmetrical case $\mu_1+r\mu_2^- >0$ ensures that $Z_1(t)\to \infty$. It can be treated in the same way.}
\end{equation}
ensuring that $Z_1(t)\to -\infty$ as $t\to\infty$ (see Appendix \ref{AppendixB}, Lemma~\ref{lem:transient}). Throughout this work, our primary concern shall be the case where
\begin{equation} 
\mu_2< 0.
\label{eq:conddriftneg} \footnote{See Appendix~\ref{subsec:driftpos} for the case $\mu_2\geqslant 0$.}
\end{equation}
Under \eqref{eq:conddriftneg}, $\ell(t)\to \infty$, $\mu_2^-=-\mu_2$, and \eqref{eq:conditiontrans} is equivalent to $r\mu_2-\mu_1>0$.  Figure \ref{fig:rebondderive} below gives  two examples of parameters satisfying \eqref{eq:conditiontrans} and \eqref{eq:conddriftneg}.

\begin{figure}[hbtp]
\centering
\includegraphics[scale=0.5]{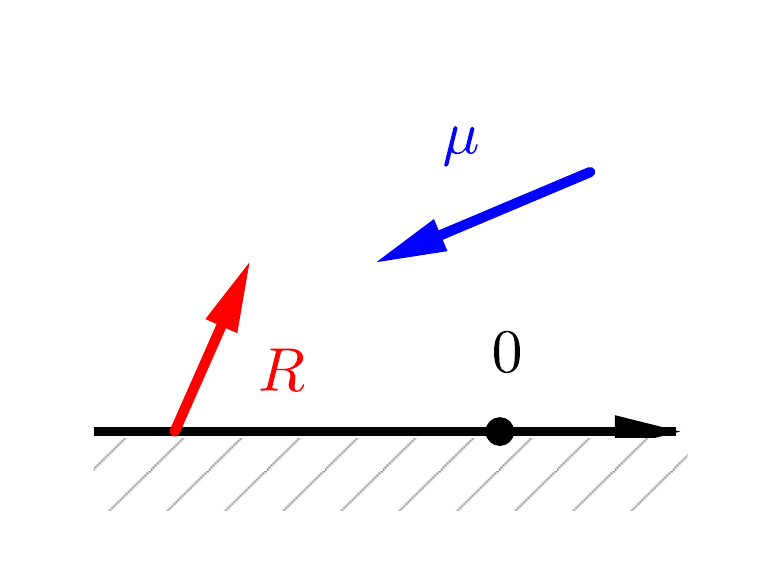}
\includegraphics[scale=0.5]{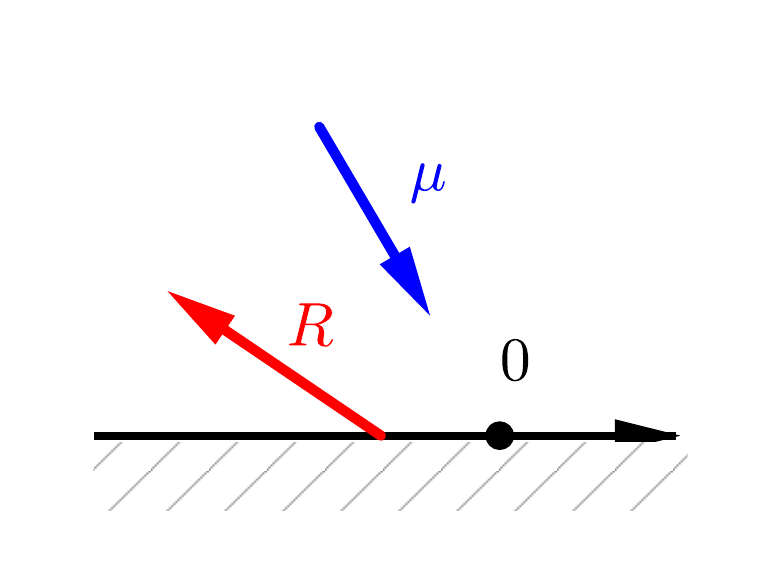}
\caption{Two examples of parameters satisfying the inequality in \eqref{eq:conditiontrans} and \eqref{eq:conddriftneg}. $\mu$ is the drift and $R$ is the reflection vector}
\label{fig:rebondderive}

\end{figure}

Let $p_t(z)$ denote the density function of the random vector $Z(t)$ at the point $z$ in the upper half-plane. For any bounded set $A$, define
\begin{equation}
\pii(z):=\int_0^\infty p_t (z) \mathrm{d} t,
\label{def:greenfunction}
\end{equation}
and
\begin{equation*}
\Pi(A):=\int_A \pii(z) \mathrm{d} z = \mathbb{E} \left[ \int_0^\infty \mathbf{1}_{A} (Z(t)) \mathrm{d}t \right].
\end{equation*}
We call $\Pi$ the \textit{Green's measure} of the process $Z$ and $\pii$ the \textit{occupancy density} (alternatively, the Green's function) of the process $Z$. 
Let $(\rho,\alpha)$ be the polar coordinate representation of a point $z$ in the upper half-plane.
\, The \textit{occupancy measure} on the boundary (alternatively, the ``pushing measure'' or the ``Green's measure'') is defined as
$$\nu(A):=
\mathbb{E} \left[ \int_0^\infty \mathbf{1}_{A} (Z(t)) \mathrm{d}\ell(t) \right].
$$
Notice that $\ell$ increases only when $Z_2(t)=0$, which corresponds to the support of $\nu$ lying on the abscissa. Indeed, $\nu$ is \rems{a Borel} measure and has density with respect to Lebesgue measure on the abscissa (see Harrison and Williams \cite[\S 8]{harrison_brownian_1987}). In particular, let $\nu_1$ be the density such that $\nu(\mathrm{d} z)=\nu_1(z_1)\mathrm{d}z_1 \times \delta_0(\mathrm{d}z_2)$. 
\\
\indent With the above preparations now in hand, we now state Harrison's \textit{open problem}.\\\\
\noindent \textbf{Harrison's Problem \cite{Web}:} 
Determine the exact asymptotic behavior of $\pii(\rho,\alpha)$ with $\rho \rightarrow \infty$ and $\alpha$ fixed.\\\\ 
\indent Theorem  \ref{thm:main} of this paper closes this problem. In the process of finding the exact asymptotic behavior of $\pii(\rho,\alpha)$ with $\rho \rightarrow \infty$ and $\alpha$ fixed, we also determine the exact tail asymptotic behavior of the boundary occupancy measure $\nu$ (Proposition \ref{prop:asymptnu}) and an explicit integral expression for the occupancy density $\pii$ (Proposition~\ref{lem:inverselaplace}). {These asymptotics} lead us to explicitly determine all harmonic functions of the Martin compactification and to obtain the Martin boundary of the process (Proposition \ref{prop:martin}).

The significance of Harrison's problem is directly related to the task of finding the exact asymptotic behavior of the stationary density of RBM in a quadrant. 
Referring to this task, 
Harrison remarks that ``given the `cones of boundary influence' discovered by Avram, Dai and Hasenbein
\cite{avram_explicit_2001}, one may plausibly hope to crack the problem by piecing together the asymptotic analyses of occupancy densities for three much simpler processes: a RBM in the upper half-plane that is obtained by removing the left-hand boundary of the quadrant; a RBM in the right half-plane that is obtained by removing the lower boundary of the quadrant; and the unrestricted Brownian motion that is obtained by removing \textit{both} of the quadrant's boundaries.'' (\cite{Web}). Harrison further emphasizes the importance of the problem at hand by writing that ``at the very least, the solution of the problem posed above may provide a deeper understanding or alternative interpretation of recent results on the asymptotic behavior of various quantities associated with the stationary distribution of RBM in a quadrant,'' as in \cite{dai_reflecting_2011,dai_stationary_2013,franceschi_asymptotic_2016}.
\rems{The exact asymptotics of the stationary distribution for RBM in a quadrant
were recently determined 
in \cite{franceschi_asymptotic_2016}. The present article provides progress towards understanding many of the missing pieces in both \cite{avram_explicit_2001} and \cite{franceschi_asymptotic_2016}. The present results may also be used to investigate the consistency of the asymptotics obtained in \cite{franceschi_asymptotic_2016} with the analysis of \cite{avram_explicit_2001}.}
%
%
\\
\indent The tools in this paper are, in part, inspired by methods introduced by the seminal work of Malyshev \cite{malyshev_asymptotic_1973}, which studies the asymptotic behavior of the stationary distribution for random walks in the quadrant. Subsequent works studying asymptotics in the spirit of Malshev's approach include \cite{kurkova_martin_1998}, which studies  the Martin boundary of random walks in the quadrant and in the half-plane; \cite{kurkova_malyshevs_2003}, which extends the methods of Malyshev to the join-the-shorter-queue paradigm; \cite{kourkova_random_2011}, which studies the asymptotics of the Green's functions of random walks in the quadrant with non-zero drift absorbed at the axes, and \cite{franceschi_asymptotic_2016}, which extends Malyshev's method to computing asymptotics in the continuous case. \rems{To the best of our knowledge, this is the first time that such a method has been employed in the continuous case (for Brownian motion) for computing a Martin boundary.}

A second group of literature closely relating to the present paper is that which concerns the asymptotics of the stationary distribution of semi-martingale reflecting Brownian motion (SRBM) in the quadrant \cite{dai_reflecting_2011,dai_stationary_2013} or in the orthant \cite{miyazawa_conjectures_2011}. 
\rems{Nonetheless, our techniques still differ from those in \cite{dai_reflecting_2011,dai_stationary_2013,miyazawa_conjectures_2011} because of our use of the saddle point method.}
These three papers develop a similar analytic method and contain similar asymptotic results to those for SRBM arising from a tandem queue \cite{lieshout_tandem_2007,lieshout_asymptotic_2008, miyazawa_tail_2009}.
\\
\indent The remainder of the paper is organized as follows. Proposition \ref{prop:functionaleq} of Section~\ref{sec:functionaleq} establishes a kernel functional equation linking 
the moment generating functions of the measures $\pii$ and $\nu$. 
Section~\ref{sec:boundarymeasure} is concerned with the boundary occupancy measure. An explicit expression for its moment generating function 
is established in Lemma~\ref{lem:MGFexplicit} and its singularities are studied. The exact tail asymptotics of $\nu$ are subsequently given in Proposition~\ref{prop:asymptnu}.
Proposition~\ref{lem:inverselaplace} of Section~\ref{sec:inverselaplace} expresses the occupancy density $\pii$ as a simple integral via Laplace transform inversion. 
Theorem~\ref{thm:main} in Section~\ref{sec:saddlepoint} provides the paper's key result on the exact asymptotic behavior of $\pii (\rho,\alpha)$ as $\rho\to \infty$ with $\alpha\in(0,\pi)$. Section~\ref{sec:martin} is devoted to the study of the Martin boundary and to the corresponding harmonic functions.







\section{A kernel functional equation}
\label{sec:functionaleq}
\noindent We begin by defining the moment generating function (MGF) (alternatively, bilateral Laplace transform) of the measures $\pii$ and $\nu$. For $\theta=(\theta_1,\theta_2)\in \mathbb{C}^2$, let
$$
f(\theta):=\hat{\pii}(\theta)= \int_{\mathbb{R}\times \mathbb{R}_+} e^{\theta \cdot z} \pii(z)\mathrm{d}z
=\mathbb{E} \left[ \int_0^{\infty}e^{\theta \cdot Z(s)} \mathrm{d} s \right],
$$
and
$$
g(\theta_1):=\hat{\nu}(\theta)=\hat{\nu_1}(\theta_1)= \int_{\mathbb{R}} e^{\theta_1 \cdot z_1} \nu_1(z_1)\mathrm{d}z_1=\mathbb{E} \left[ \int_0^{\infty}e^{\theta \cdot Z(s)} \mathrm{d}  \ell(s) \right].
$$
We note that $g$ depends only on $\theta_1$;  it does not depend on $\theta_2$ since the support of  $\nu$ lies on the abscissa. Further, $f$ is a two-dimensional Laplace transform which is bilateral for one dimension. We wish to establish a \textit{kernel functional equation} linking the moment generating functions $f$ and $g$ (Proposition \ref{prop:functionaleq}).\\
\indent Consider the kernel
\begin{equation}
Q(\theta):=\frac{1}{2} (\theta_1^2+\theta_2^2)+\mu_1\theta_1+\mu_2\theta_2=\frac{1}{2}(|\theta+\mu|^2-(\mu_1^2+\mu_2^2)).
\label{eq:def:Q}
\end{equation}
Note that $Q(\theta)t=\log \mathbb{E}[e^{\theta\cdot \rems{(B(t)+\mu t)}}]$ is the cumulant-generating function of $\rems{B(t)+\mu t}$. The kernel $Q$ is also called the ``characteristic exponent'' or the ``Lévy exponent'' of $\rems{B(t)+\mu t}$.
Let $\Theta_2^\pm (\theta_1)$ denote the functions which ``cancel'' the kernel, i.e. the functions $Q(\theta_1,\Theta_2^\pm (\theta_1))=0$. This yields
\begin{equation}\label{analytic}
\Theta_2^\pm (\theta_1):=-\mu_2\pm\sqrt{(\mu_1^2+\mu_2^2)-(\theta_1+\mu_1)^2}\,\,,
\end{equation}
where 
\begin{equation}
\theta_1^\pm:=-\mu_1\pm \sqrt{\mu_1^2+\mu_2^2}\,,
\end{equation}
denotes the points which cancel the quantity under the square root. It is evident that \eqref{analytic} is analytic on $\mathbb{C}\setminus ((-\infty,\theta_1^-]\cup[\theta_1^+\infty))$. 
\rems{Let us define $$\theta^+:=(\theta_1^+,\Theta_2^+(\theta_1^+)))=(\theta_1^+,-\mu_2).$$}
Note also that $\theta_1^+>0$ and that $\theta_1^-<0$.
Let $$\theta^p_1 :=\frac{2(r\mu_2-\mu_1) }{r^2+1} \in (0,\theta_1^+), $$
be the first coordinate of the point of intersection between the circle $\rems{Q (\theta)}=0$ and the line $R\cdot \theta=0$ (see Figure~\ref{fig:circleline} below).
\begin{figure}[hbtp]
\centering
\includegraphics[trim={0.5cm 0.8cm 0.5cm 0.5cm}, clip,scale=3]{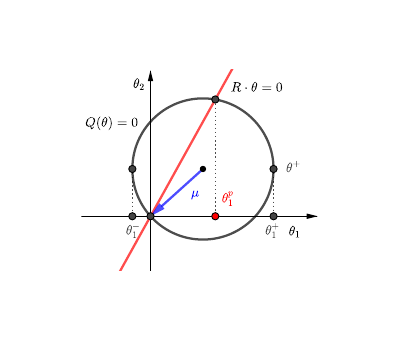}
\caption{Circle $Q(\theta)=0$, line $R\cdot \theta =0$ and points $\theta_1^\pm$ and $\theta_1^p$.}
\label{fig:circleline}
\end{figure}

\noindent
Define the sets
$$
E:= \{
\theta\in\mathbb{C}^2 :
\exists\,\widetilde\theta \in \mathbb{R}^2 \text{ such that } \widetilde\theta_1=\Re\theta_1, \ \Re\theta_2\leqslant \widetilde\theta_2, \ \widetilde\theta \cdot R<0, \text{ and } Q(\widetilde\theta)<0
\}.
$$
and
$$
F:=\{\theta\in\mathbb{C}^2 : 0<\Re \theta_1<\theta_1^p \text{ and } \Re\theta_2\leqslant 0 \}.
$$ Figure~\ref{fig:E} below provides a visual representation of $E \cap \mathbb{R}^2$ and $F \cap \mathbb{R}^2$.

\begin{figure}[H]
\centering
\includegraphics[scale=0.2]{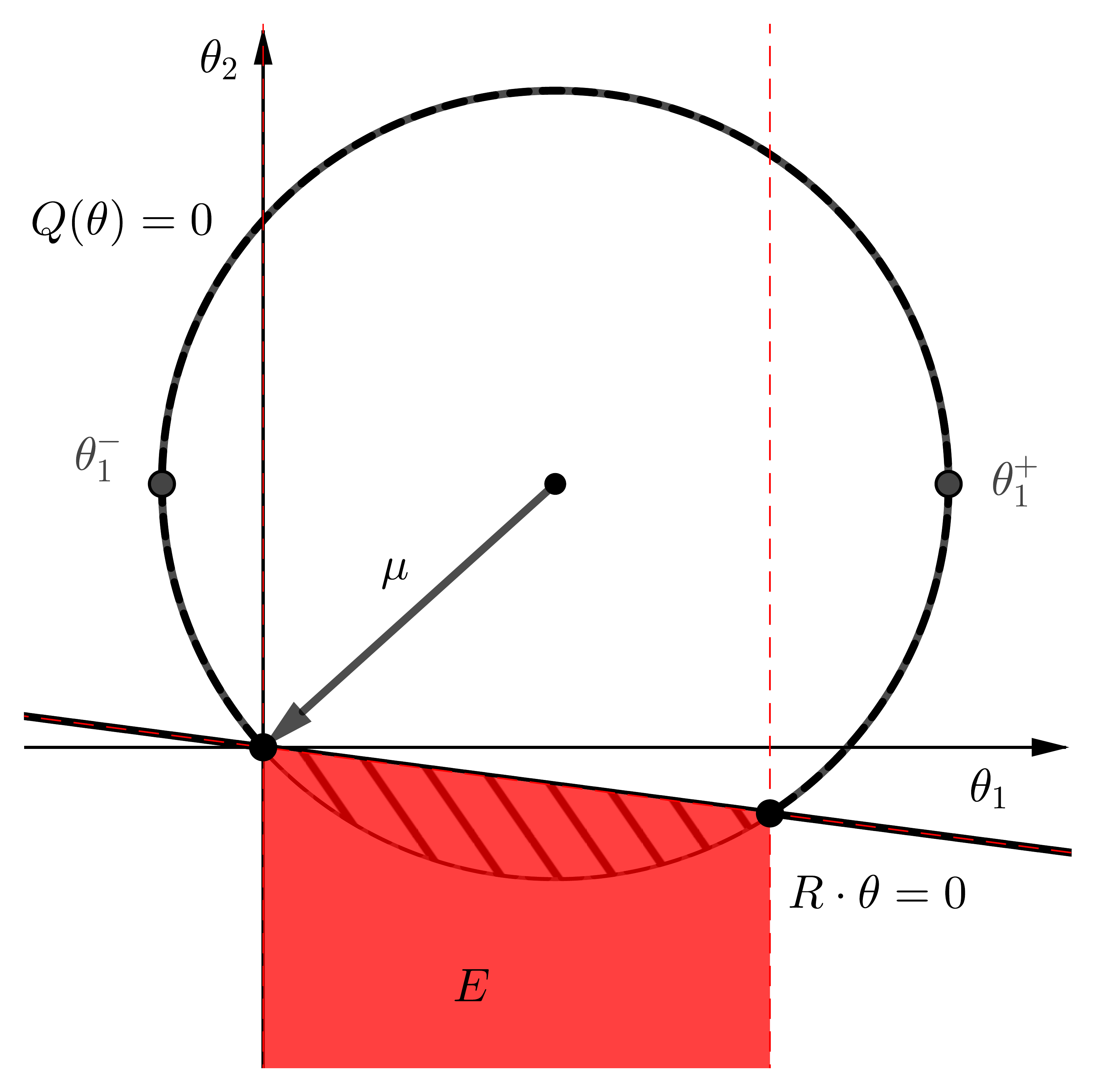}
\includegraphics[scale=0.2]{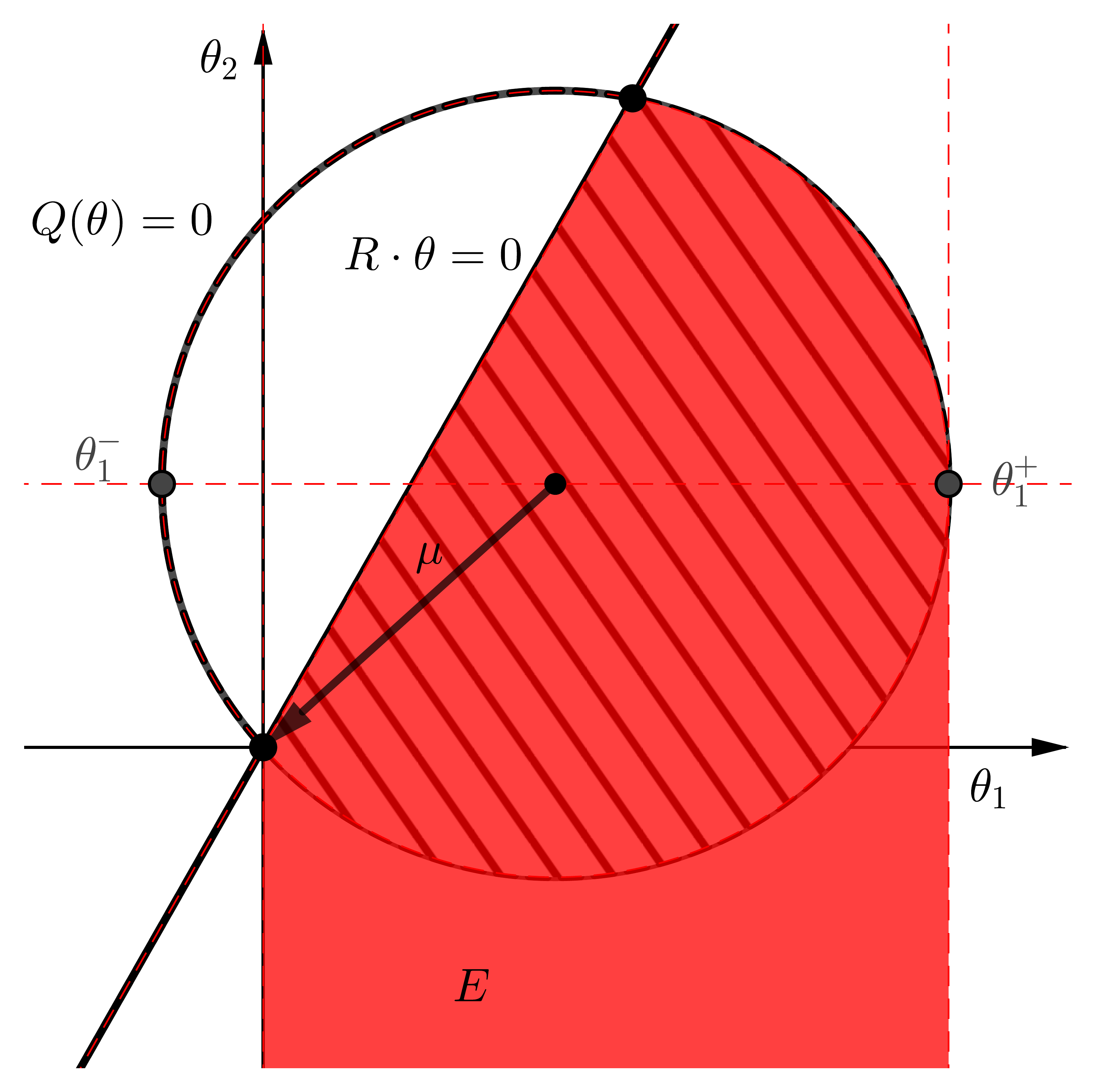}
\includegraphics[scale=0.2]{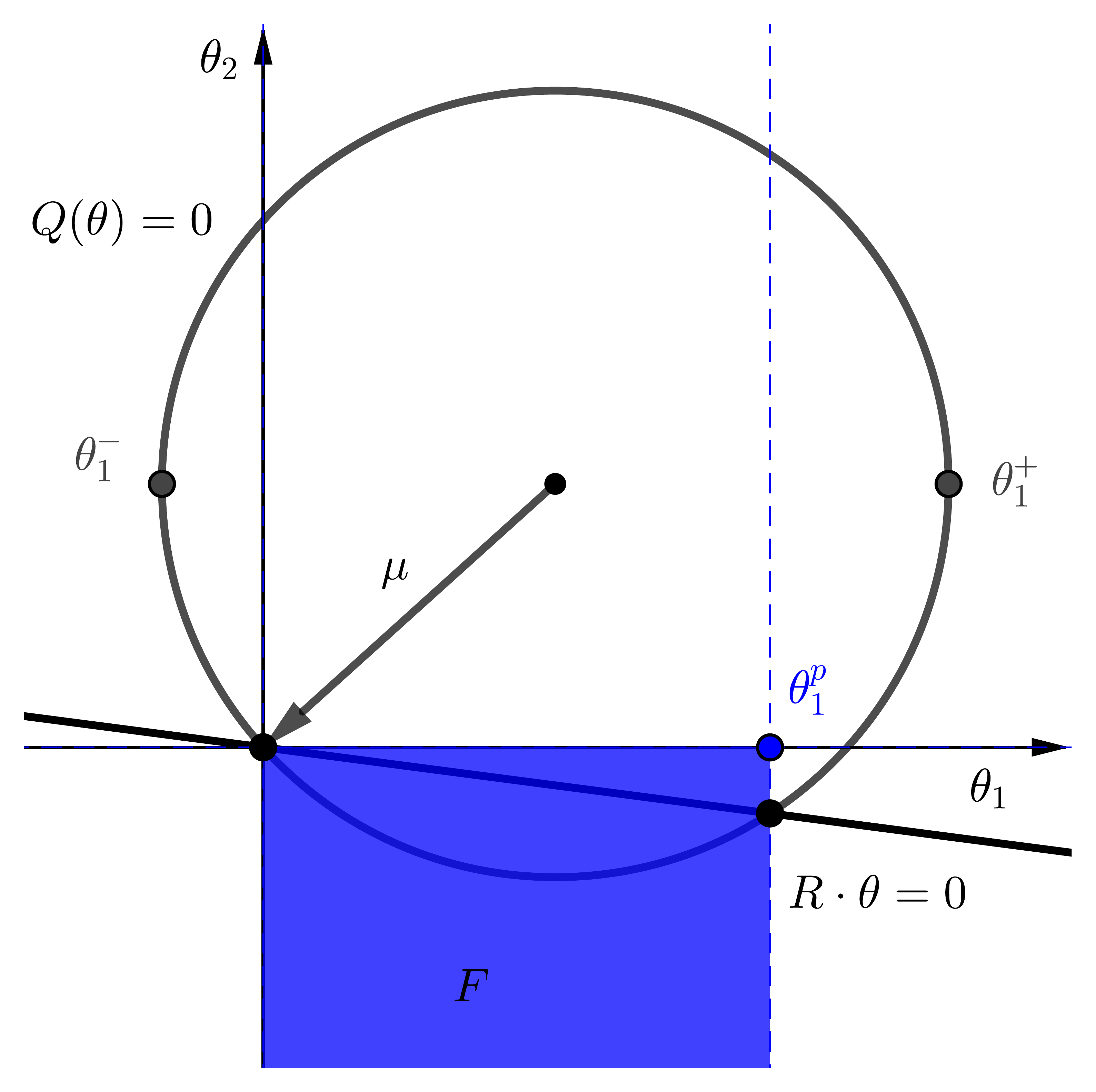}
\includegraphics[scale=0.2]{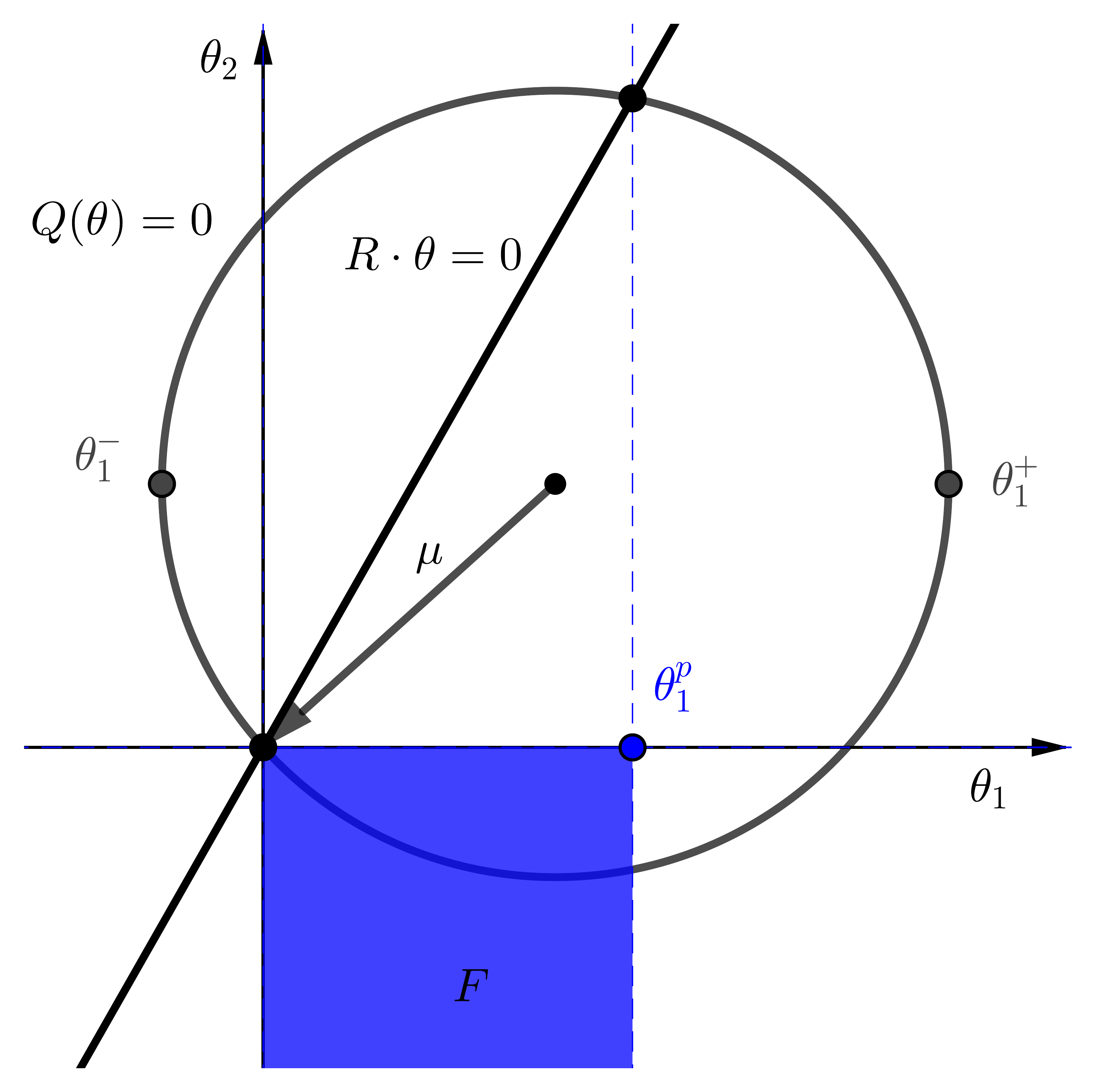}
\caption{In the first two pictures, the domain $E\cap\mathbb{R}^2$ is colored in red and the hatched subset of $E\cap\mathbb{R}^2$ is the set $\{\theta\in\mathbb{R}^2 :R \cdot\theta <0 \text{ and } Q(\theta)<0 \}$. In the last two pictures, the domain $F\cap\mathbb{R}^2$ is colored in blue. In the two pictures on the left, $r>0$ and $E\subset F$. In the two pictures on the right, $r<0$ and $F\subset E$.}
\label{fig:E}
\end{figure}
\noindent We now turn to studying the domains of convergence for $f$ and $g$.
\begin{lem}
For $\theta\in E\cup F$ we have that
\begin{equation}
 \lim_{t \rightarrow \infty} \mathbb{E} [e^{\theta \cdot Z(t)} ]=0. \label{eq17}
\end{equation}
Further,
\begin{equation}
f(\theta)=\mathbb{E} \left[ \int_0^{\infty}e^{\theta \cdot Z(s)} \mathrm{d} s \right] < \infty
\quad \text{and} \quad
g(\theta_1)=\mathbb{E} \left[ \int_0^{\infty}e^{\theta \cdot Z(s)} \mathrm{d}  \ell(s) \right]<\infty.
 \label{eq18}
\end{equation}
\label{lem:limitfinite}
\end{lem}
\begin{proof}
We consider the two cases $\theta\in E$ and $\theta\in F$ separately below.
\begin{enumerate}[label=(\roman*)]
\item
\underline{Let $\theta\in E$.} Consider $\widetilde\theta$ satisfying the conditions stated in the definition of the set $E$, that is $ \widetilde\theta_1=\Re\theta_1, \ \Re\theta_2\leqslant \widetilde\theta_2, \ \widetilde\theta \cdot R<0, \text{ and } Q(\widetilde\theta)<0$. We have
\begin{align*}
\mathbb{E} [|e^{\theta \cdot Z(t)}| ]=\mathbb{E} [e^{\Re\theta \cdot Z(t)} ] &\leqslant \mathbb{E} [e^{\widetilde\theta \cdot Z(t)}], \,\,(\text{since } \Re\theta_2 \leqslant \widetilde{\theta}_2 \text{ and } Z_2(t)\geqslant 0),
\\ &\leqslant \mathbb{E} [e^{\widetilde\theta \cdot (B(t)+\mu t)+(\widetilde\theta \cdot R)\ell(t)}],
\\ &\leqslant \mathbb{E} [e^{\widetilde\theta \cdot (B(t)+\mu t)}], \,\, (\text{since} \,\,\widetilde\theta \cdot R <0 \text{ and } \ell(t)\geqslant 0),
\\ &\leqslant e^{Q(\widetilde\theta )t},
\text{ (the MGF of a Gaussian)}
\\ \text{and then }\mathbb{E} [e^{\theta \cdot Z(t)} ] & \underset{t\to\infty}{\longrightarrow} 0 \text{ for } Q(\widetilde\theta )<0.
\end{align*}
From the inequality $\mathbb{E} [e^{\theta \cdot Z(t)} ]\leqslant e^{Q(\widetilde\theta )t}$ and by Fubini's theorem,
$\mathbb{E} \left[ \int_0^{\infty}e^{\theta \cdot Z(s)} \mathrm{d} s \right] < \infty$. Letting $t$ tend to infinity in equation~\eqref{eq:democonvergence}, we easily obtain that $\mathbb{E} \left[ \int_0^{\infty}e^{\theta \cdot Z(s)} \mathrm{d}  \ell(s) \right]<\infty$.\\
\item \underline{Let $\theta\in F$.} Let $a:= \Re\theta_1$. Noting that $Z_{2} (t)$ is non-negative for every $t \geq 0$ and $\Re\theta_2 \leq 0$, we have
\begin{eqnarray*}
\left|e^{{\theta}\cdot {Z}(t)} \right| = \left| e^{\theta_1 Z_1(t) + \theta_2 Z_{2}(t)} \right|
\leq e^{\Re\theta_1 Z_1(t) + \Re\theta_2 Z_{2}(t)}
\leq e^{a Z_1(t)}.
\end{eqnarray*}
Noting that $B_1(t)$ and $B_2(t)$ are assumed independent, and employing the inequality in \eqref{eq2} of the Appendix, we have that
\begin{eqnarray*}
&& \Big|\mathbb{E}\left[ e^{{\theta}\cdot {Z}(t)} \right] \Big|
\leq \mathbb{E}\left[ e^{a Z_1(t)} \right] \\
&\leq& \mathbb{E}\left[ e^{a \left((\mu_{1} + r \mu_{2}^{-})t + B_{1}(t) + |r| \sup_{0 \leq s \leq t} |B_{2}(s)|\right)  } \right] \\
&=& e^{a (\mu_{1} + r \mu_{2}^{-})t} \cdot \mathbb{E}\left[ e^{a B_{1}(t)} \right] \cdot
\mathbb{E}\left[e^{a |r| \sup_{0 \leq s \leq t} |B_{2}(s)|}\right] \\
&=& e^{a (\mu_{1} + r \mu_{2}^{-})t} \cdot e^{\frac{1}{2}\,a^2 t} \cdot \mathbb{E}\left[\sup_{0 \leq s \leq t} e^{a |r|  |B_{2}(s)|}\right].
\end{eqnarray*}
Since $ x \mapsto \exp(a|r||x|)$ is a convex function, $\exp(a|r| |B_{2}(t)|)$ is a submartingale. By Doob's $L^2$ Maximal Inequality, we have
\begin{eqnarray*}
&&\left( \mathbb{E} \left( \sup_{0\leq s \leq t} e^{a|r| |B_{2}(s)|} \right) \right)^{\frac{1}{2}} \leq 2 \left( \mathbb{E} e^{a|r| |B_{2}(t)|} \right)^{\frac{1}{2}}  \\
&\leq& 2 \left( \rems{ \mathbb{E} e^{a|r| B_{2}(t)}  + \mathbb{E} e^{- a|r| B_{2}(t)} } \right)^{\frac{1}{2}} = 2\sqrt{2} \, e^{\frac{1}{4} a^2 r^2 t}.
\end{eqnarray*}
Thus
\begin{equation*}
 \mathbb{E}\left( \sup_{0 \leq s \leq t} e^{a|r| |B_{2}(s)|} \right)
 \leq 8 \, e^{\frac{1}{2} a^2 r^2 t},   \label{eq12}
\end{equation*}
and
\begin{equation}
\Big|\mathbb{E}\left[ e^{{\theta}\cdot {Z}(t)} \right] \Big|
\leq 8 \, e^{\left(a \left(\mu_1 + r \mu_2^{-} \right) + \frac{1}{2}\,a^2 + \frac{1}{2}\,a^2 r^2\right) t}. \label{eq19}
\end{equation}
Since $\theta\in F$, we have $0<a<\theta_1^p=\frac{2(r\mu_2-\mu_1) }{r^2+1}$ and 
\begin{equation*}
a \left(\mu_1 + r \mu_2^{-} \right) + \frac{1}{2}\,a^2 + \frac{1}{2}\,a^2 r^2 <0.
\end{equation*}
Equation \eqref{eq17} now follows immediately from the inequality in \eqref{eq19}. The first statement of convergence in \eqref{eq18} follows from the inequality in \eqref{eq19} and by Fubini's theorem. As in the case $\theta\in E$, we conclude the proof letting $t$ go to infinity in equation~\eqref{eq:democonvergence}. The second statement of convergence in \eqref{eq18} then immediately follows.
\\
\noindent

\end{enumerate}
\end{proof}
We now turn to Proposition \ref{prop:functionaleq}, which provides a kernel functional equation linking the functions $f$ and $g$.

\begin{prop}
For all $\theta=(\theta_1,\theta_2)$ in the set $E\cup F$,
the integrals $f(\theta)$ and $g(\theta_1)$ are finite and the following functional equation holds
\begin{equation}
0=1+ Q(\theta)f(\theta) + (R \cdot \theta) g(\theta_1),
\label{eq:functional_eq}
\end{equation}
where $Q$ is the kernel defined in \eqref{eq:def:Q}.
\label{prop:functionaleq}
\end{prop}

\begin{proof}
For 
 $f\in \mathcal{C}^2 (\mathbb{R}\times\mathbb{R}_+)$, we have by It\^o's Lemma that
\begin{equation}\label{Ito}
 f(Z(t))-f(Z(0))=\int_0^t \nabla f(Z(s)).\mathrm{d} B(s) + \int_0^t \mathcal{L}f(Z(s)) \mathrm{d} s +  \int_0^t  R \cdot \nabla f(Z(s)) \mathrm{d} \ell(t),
\end{equation}
where $\mathcal{L}$ is the generator
$$
\mathcal{L}=\frac{1}{2} \triangle + \mu.\nabla
.
$$
For $z\in\mathbb{R}\times\mathbb{R}_+$, we shall let $f(z)=e^{\theta \cdot z}$. We proceed to take expectations of the equality in \eqref{Ito}. The integral 
$\int_0^t \nabla f(Z(s)).\mathrm{d} B_s $ is a martingale 
and thus its expectation is zero. This yields
\begin{equation}
\mathbb{E} [ e^{\theta \cdot Z(t)} ] - 1
=
0 
+ 
Q (\theta)   \mathbb{E} \left[ \int_0^t  e^{\theta\cdot Z(s)} \mathrm{d} s \right]
+ 
(R \cdot \theta)  \mathbb{E} \left[  \int_0^t e^{\theta\cdot Z(s)} \mathrm{d} \ell(s) \right].
\label{eq:democonvergence}
\end{equation} 
We now invoke Lemma~\ref{lem:limitfinite}. For $\theta\in E\cup F$, $\mathbb{E} [e^{\theta \cdot Z(t)} ] \underset{t\to\infty}{\longrightarrow} 0 $. Further, by Lemma~\ref{lem:limitfinite}, the integrals 
$\mathbb{E} \left[ \int_0^{\infty}e^{\theta \cdot Z(s)} \mathrm{d}  \ell(s) \right]$ and $\mathbb{E}\left[ \int_0^{\infty}e^{\theta \cdot Z(s)} \mathrm{d} s \right]$ are finite. Letting $t$ tend to infinity in equation \eqref{eq:democonvergence}, we obtain 
\begin{equation*}
0 - 1
=
Q (\theta)   \mathbb{E} \left[ \int_0^{\infty}e^{\theta \cdot Z(s)} \mathrm{d} s \right]
+ 
(R\cdot \theta)  \mathbb{E} \left[ \int_0^{\infty}e^{\theta \cdot Z(s)} \mathrm{d}  \ell(s) \right],
\end{equation*}
which indeed is equation \eqref{eq:functional_eq}. This concludes the proof.
\end{proof}

We shall use the convergence of $f$ and $g$ on the set $E$ in the proof of Lemma~\ref{lem:MGFexplicit}. The convergence on the set $F$ will be employed in the proof of Proposition~\ref{lem:inverselaplace}.

\section{Boundary occupancy measure}
\label{sec:boundarymeasure}
\noindent This section concerns the study of the boundary occupancy measure. We shall find an explicit expression for its MGF in Lemma~\ref{lem:MGFexplicit} and Proposition~\ref{prop:asymptnu} provides its exact asymptotics. Throughout,  denote $\sqrt{\cdot}$ to be the principal square root function which is analytic on $\mathbb{C}\setminus (-\infty, 0]$ and such that $\sqrt{1}=1$.
\begin{lem}
The moment generating function of the boundary occupancy measure can be meromorphically continued to the set $\mathbb{C}\setminus ((-\infty,\theta_1^-] \cup [\theta_1^+,\infty))$ and is equal to
\begin{equation}
g(\theta_1)
=\frac{-1}{r\theta_1+\Theta_2^- (\theta_1)}
=\frac{1}{-r\theta_1+\mu_2+\sqrt{(\mu_1^2+\mu_2^2)-(\theta_1+\mu_1)^2}},
\label{eq:value:g}
\end{equation}
for all $\theta_1\in\mathbb{C}\setminus ((-\infty,\theta_1^-] \cup [\theta_1^+,\infty))$. The function $g$ then has a simple pole at $0$ and has another pole in $\mathbb{C}\setminus ((-\infty,\theta_1^-] \cup [\theta_1^+,\infty))$ if and only if 
\begin{equation}
\rems{R\cdot \theta^+}
=r \theta_1^+ -\mu_2
>0.
\label{eq:CNSpole}
\end{equation}
When it exists, the other (simple) pole is $$\theta^p_1 :=\frac{2(r\mu_2-\mu_1) }{r^2+1} \in (0,\theta_1^+) .$$
Finally, in the neighborhood of $\theta_1^+$,
\begin{equation*}
g (\theta_1) 
\rems{=}
\begin{cases}
\frac{1}{-r\theta_1^+ +\mu_2}- \frac{1}{(-r\theta_1^+ +\mu_2)^2} {\sqrt{(\theta_1^+ -\theta_1)(\theta_1^+-\theta_1^-)}}  
+ O(\theta_1-\theta_1^+) & \text{if } r \theta_1^+ -\mu_2 \neq 0,
\\
\frac{1}{\sqrt{(\theta_1^+ -\theta_1)(\theta_1^+-\theta_1^-)}} + O(1) & \text{if } r \theta_1^+ -\mu_2 = 0\,.
\end{cases}
\end{equation*}
\label{lem:MGFexplicit}
\end{lem}
\begin{proof}
For $\epsilon>0$, let us denote $\widetilde{\theta} = (\theta_1,\Theta_2^- (\theta_1)+\epsilon)$. One may easily verify that for both $\theta_1>0$ sufficiently small and $\epsilon>0$ sufficiently small we have that $\rems{R\cdot \widetilde{\theta}}<0 $ and $Q(\widetilde{\theta})<0$. Together, these inequalities imply that $(\theta_1,\Theta_2^- (\theta_1))\in E$. $E$ is an open set and by continuity we have that $(\theta_1,\Theta_2^- (\theta_1))\in E$ for $\theta_1$ in some open non-empty set. \\
\indent We now evaluate the functional equation~\eqref{eq:functional_eq} at the points $(\theta_1,\Theta_2^- (\theta_1))\in E$. Since 
\begin{equation*}
Q(\theta_1,\Theta_2^- (\theta_1))=0,
\end{equation*}
equation~\eqref{eq:value:g} is satisfied for $\theta_1$ in some open non-empty set. By the principle of analytic continuation, we may continue $g$ on the set $\mathbb{C}\setminus ((-\infty,\theta_1^-] \cup [\theta_1^+,\infty))$, the latter being the domain of the function in equation \eqref{eq:value:g}. The square root at the denominator of this function can be written as 
\begin{equation*}
\sqrt{(\theta_1^+-\theta_1)(\theta_1-\theta_1^-)}.
\end{equation*}
We emphasize have taken the principal square root function with a cut on $(-\infty, 0]$ and such that $\sqrt{1}=1$. \\
\indent The remainder of the proof proceeds in a straightforward manner.
Finding the poles of the function 
\begin{equation*}
\frac{1}{-r \theta_1 + \mu_2 + \sqrt{(\mu_1^2 + \mu_{2}^{2}) - (\theta_1 + \mu_1)^2}},
\end{equation*}
in $\mathbb{C}\setminus ((-\infty,\theta_1^-] \cup [\theta_1^+,\infty))$ is equivalent to finding the zeros of the function $-r \theta_1 + \mu_2 + \sqrt{(\mu_1^2 + \mu_{2}^{2}) - (\theta_1 + \mu_1)^2}$, i.e. solving for $\theta_1$ in the equation
\begin{equation*}
r\theta_1 - \mu_2 = \sqrt{(\mu_1^2 + \mu_{2}^{2}) - (\theta_1 + \mu_1)^2}\,.
\end{equation*}
The above equation is equivalent to the following equations
\begin{eqnarray}
&& (r\theta_1 - \mu_2)^2 = \left( \mu_1^2 + \mu_2^2 \right) - (\theta_1 + \mu_1)^2, \label{eq:root:1}\\
&& \Re (r \theta_1 - \mu_2) > 0. \label{eq:root:2}
\end{eqnarray}
The inequality in \eqref{eq:root:2} follows because the branch we select for $\sqrt{(\mu_1^2 + \mu_{2}^{2}) - (\theta_1 + \mu_1)^2}$ will ensure that the real part of  $r \theta_1 - \mu_2$ is positive. The roots of \eqref{eq:root:1} are $\theta_1 = 0$ and
\begin{equation*}
\theta_1=2(r\mu_2 - \mu_1)/(r^2 +1).
\end{equation*}
Together with \eqref{eq:root:2}, we see that $\theta_1 = 0$ is a pole of $g$ because we assumed that $\mu_2 <0$. Further,
$
\theta_1 = 2(r\mu_2 - \mu_1)/(r^2 +1)
$
is a pole of $g$ if and only if 
\begin{equation}
\left(r^2-1\right)\mu_2 - 2r \mu_1 >0.
\label{eq:CNSpole1}
\end{equation}
Under conditions \eqref{eq:conditiontrans} and \eqref{eq:conddriftneg}, it is straightforward to see that \eqref{eq:CNSpole1} is equivalent to \eqref{eq:CNSpole}. The behavior of $g$ in the neighborhood of $\theta_1^+$ is then easily obtained as desired in the statement of the Lemma.
\end{proof}

Figure~\ref{fig:polecondition} provides a geometric interpretation of the condition in \eqref{eq:CNSpole}, namely the condition for $g$ to have a pole other than $0$. The figure also illustrates the different asymptotic cases in Proposition~\ref{prop:asymptnu}. The following proposition establishes the exact asymptotics for the tail distribution of~$\nu$. 
\begin{figure}[hbtp]
\centering
\includegraphics[scale=0.14]{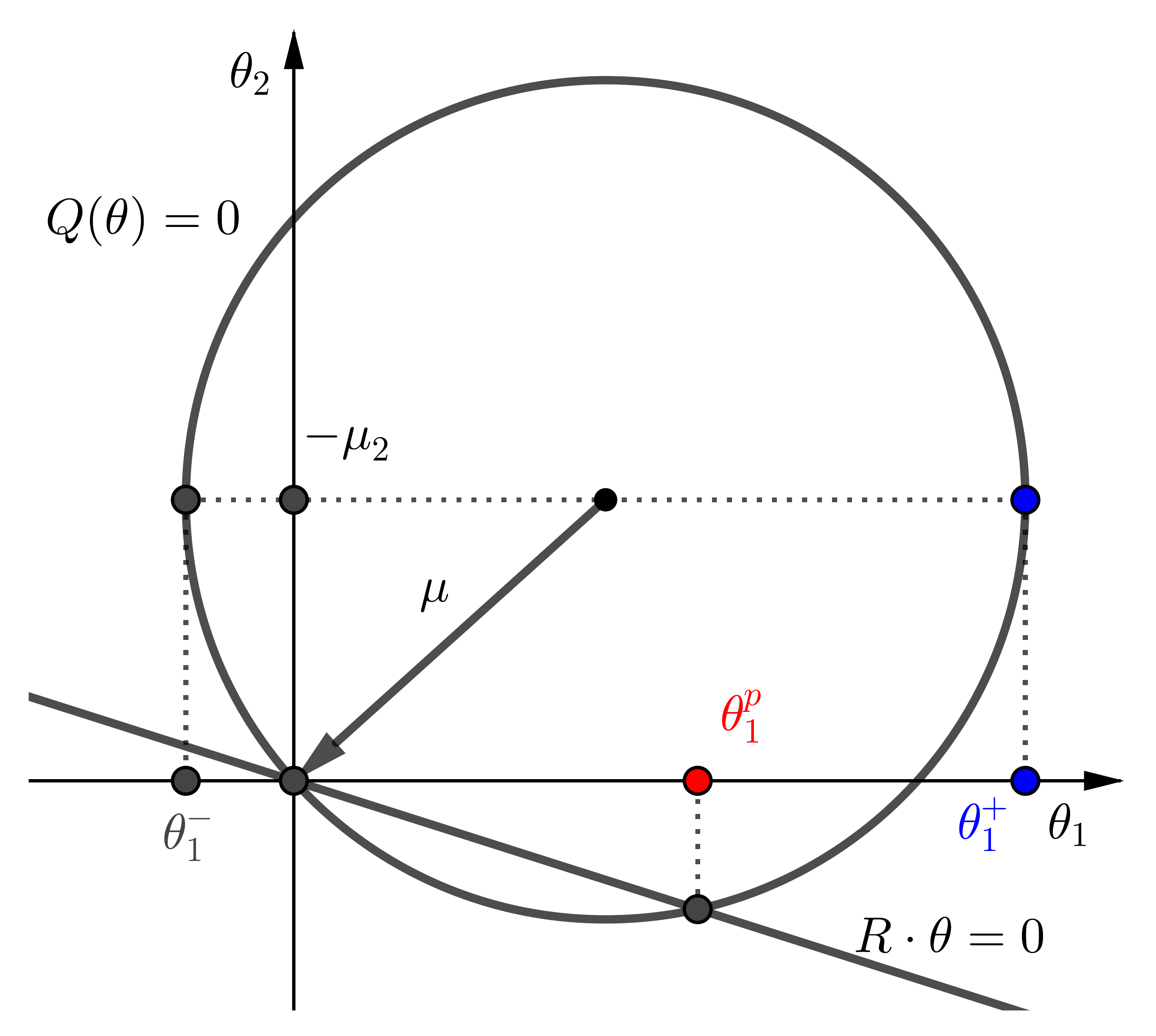}
\includegraphics[scale=0.14]{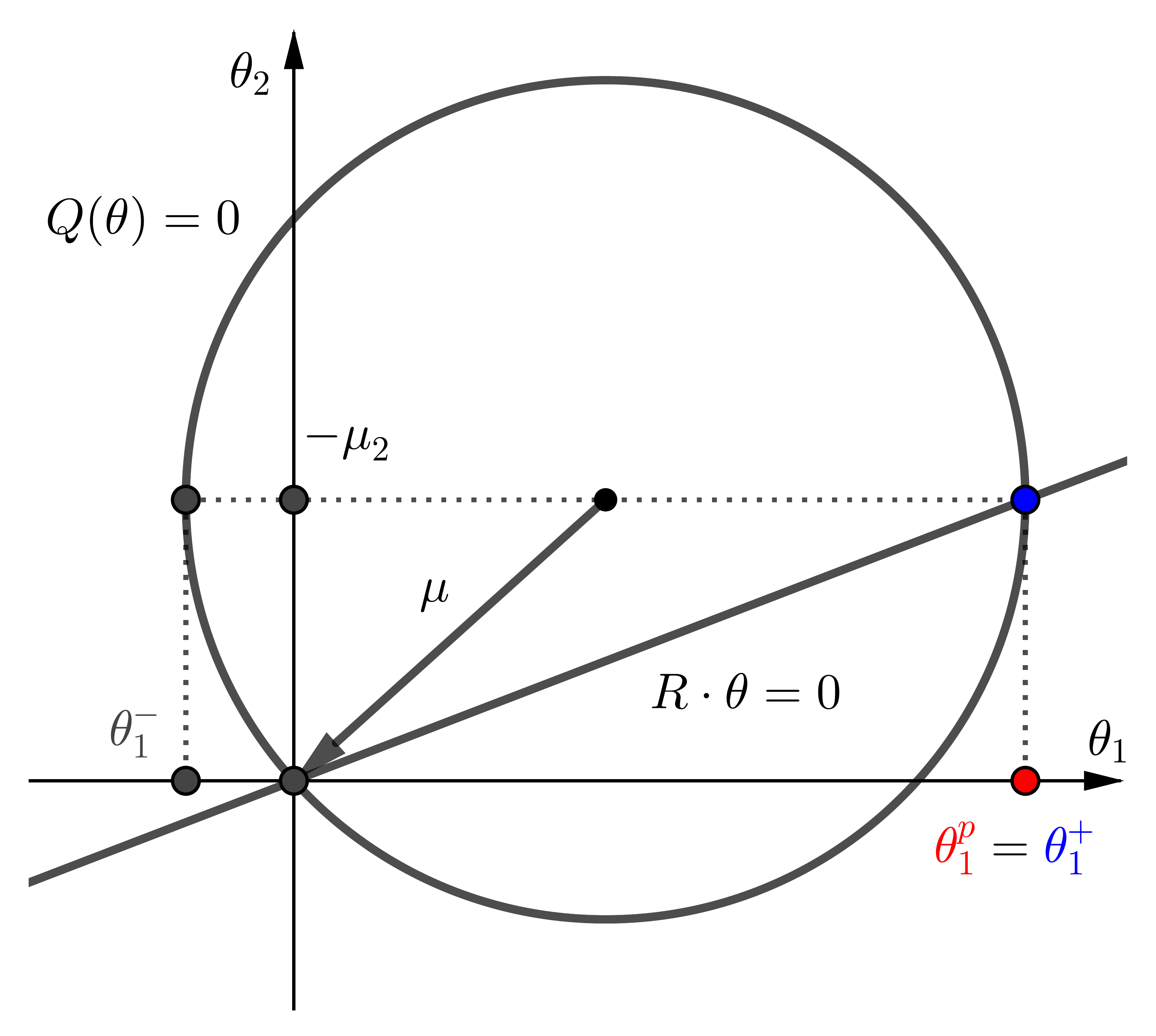}
\includegraphics[scale=0.14]{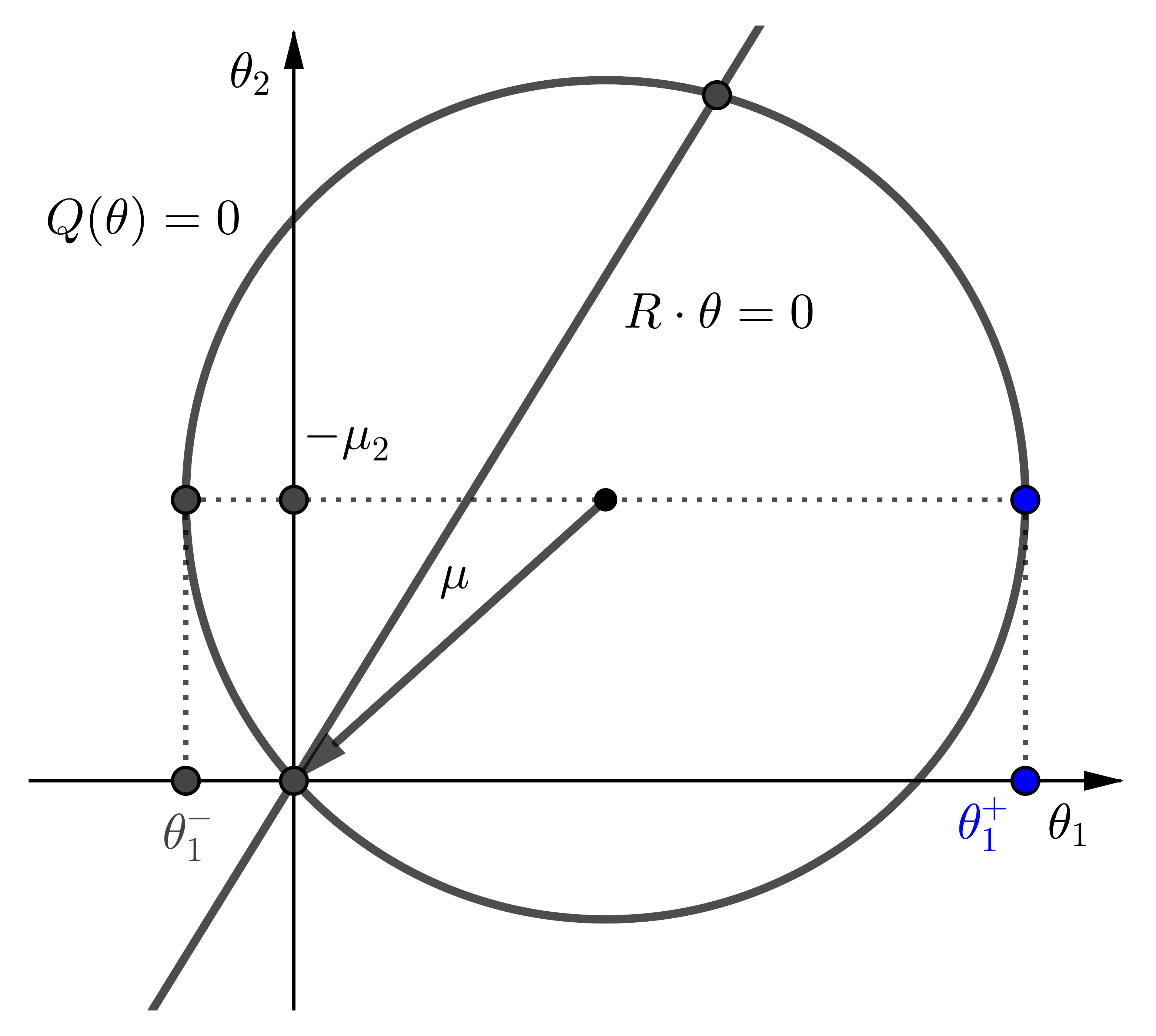}
\caption{From the left to the right: $\rems{R\cdot \theta^+}>0$, $\rems{R\cdot \theta^+}=0$, and $\rems{R\cdot \theta^+}<0$. \rems{Recall that ${R\cdot \theta^+}=r \theta_1^+ -\mu_2$.}} 
\label{fig:polecondition}
\end{figure}

\begin{prop}
The asymptotics of $\nu_1$ are given by
\begin{equation}
\nu_1 (z_1) 
\underset{z_1\to +\infty}{\sim}
\begin{cases}
A e^{-\theta_1^p z_1} & \text{if } \rems{R\cdot \theta^+}>0,
\\
B z_1^{-\frac{1}{2}} e^{-\theta_1^+ z_1} & \text{if } \rems{R\cdot \theta^+}=0,
\\
C z_1^{-\frac{3}{2}} e^{-\theta_1^+ z_1} & \text{if } \rems{R\cdot \theta^+}<0,
\end{cases}
\label{eq:asymptboudary}
\end{equation}
and by 
$$
\nu_1 (z_1) \underset{z_1\to -\infty}{\sim} D,
$$
where
\begin{equation*}
A 
=\frac{1}{r^2+1} \frac{(r^2-1)\mu_2 - 2r\mu_1}{ r\mu_2-\mu_1}, \quad
B=\frac{1}{\sqrt{\pi(\theta_1^+-\theta_1^-)}}, \quad
C=\frac{{\sqrt{(\theta_1^+-\theta_1^-)}}}{2\sqrt{\pi}(-r\theta_1^+ +\mu_2)^2},
\end{equation*}
and
\begin{equation*} D 
=\frac{\mu_2}{\mu_1 -r\mu_2}. 
\end{equation*}
The exact tail asymptotics of $\nu$, that is the asymptotics of $\nu ((z_1,\infty))$, are also given by equation~\eqref{eq:asymptboudary}, but with different constants: $A'=A/\theta_1^p$, $B'=B/\theta_1^+$ and $C'=C/\theta_1^+$.
\label{prop:asymptnu}
\end{prop}
\begin{proof}
The above results are a direct consequence of Lemma~\ref{lem:MGFexplicit} and of classical transfer theorems which link the asymptotics of a function to the singularities of its Laplace transform. These theorems rely on the complex inversion formula of a Laplace transform. For a precise statement of these theorems,
we refer the reader to \cite[Theorem 37.1]{doetsch_introduction_1974}, 
\cite[Lemma C.2]{dai_reflecting_2011} and, most importantly, to \cite[Lemmas 6.2 and 6.3]{dai_stationary_2013}, as the latter directly works with the tail distribution. The methods we shall employ to obtain the exact asymptotics for the tail distribution of boundary measures are similar in each step to those in \cite[Section~6]{dai_stationary_2013}.  \\
\indent 
Let $a$ and $b$ be the singularities which define the strip of convergence of the bilateral Laplace transform $g(\theta_1)=\int_{\mathbb{R}} e^{\theta_1 z_1} \nu(\mathrm{d}z_1)$, i.e.
\,the integral converges for $a< \Re \theta_1<b$. Note that $g$ remains defined outside this strip thanks to its analytic continuation. 
For some constants $c,c_0$, and $k$, and for $\Gamma$ the gamma function, the classical transfer theorems imply as follows:
\begin{enumerate}[label=(\roman*)]
\item If 
\begin{equation*}
g(\theta_1)-c \underset{b}{\sim} \frac{c_0}{(b-\theta_1)^k},
\end{equation*}
then
\begin{equation*}
\nu_1 (z_1)
\underset{+ \infty}{\sim} b \nu ((z_1,\infty))
 \underset{+ \infty}{\sim} \frac{c_0 }{\Gamma(k)}z_1^{k-1}e^{-b z_1}.
 \end{equation*}
\item If 
\begin{equation*}
g(\theta_1)-c \underset{a}{\sim} \frac{c_0}{(\theta_1-a)^k},
\end{equation*}
then
\begin{equation*}
\nu_1 (z_1) \underset{- \infty}{\sim}  \frac{c_0 }{\Gamma(k)}(-z_1)^{k-1}e^{-a z_1}.
\end{equation*}
\end{enumerate}
\indent \indent We now apply the consequences in (i) and (ii) above to the {study} of the singularities of $g$ in Lemma~\ref{lem:MGFexplicit}. For $r \theta_1^+ -\mu_2\leqslant 0$, the convergence strip of the integral which defines the Laplace transform has its extremities at $a=0$ and at $b=\theta_1^+$. For $r \theta_1^+ -\mu_2> 0$, the convergence strip of the integral has extremities at $a=0$ and $b=\theta_1^p$. Lemma~\ref{lem:MGFexplicit} gives
$$
g(\theta_1)\underset{0}{\sim} \frac{Res_0(g)}{\theta_1},
$$
and so $a=0$, $c_0=Res_0(g)$, $k=1$, $\Gamma(1)=1$. The transfer theorems then imply that $$\nu (z_1) \underset{z_1\to -\infty}{\sim} Res_0(g).$$
We now apply Lemma~\ref{lem:MGFexplicit} to obtain the following asymptotics in $+\infty$ for the three distinct cases given below. 
\begin{enumerate}[label=(\arabic*)]
\item If $\rems{R\cdot \theta^+}=r \theta_1^+ -\mu_2< 0$, then
$$
g(\theta_1)- \frac{1}{-r\theta_1^+ +\mu_2}
\underset{\theta_1^+}{\sim}
- \frac{1}{(-r\theta_1^+ +\mu_2)^2} {\sqrt{(\theta_1^+ -\theta_1)(\theta_1^+-\theta_1^-)}}\,,
$$
and so $b=\theta_1^+$, $c_0=- \frac{\sqrt{(\theta_1^+-\theta_1^-)}}{(-r\theta_1^+ +\mu_2)^2}  $, $k=-\frac{1}{2}$, $\Gamma(-\frac{1}{2})=-2\sqrt{\pi}$. By the transfer theorems, 
\begin{equation*}
\nu (z_1) \underset{z_1\to +\infty}{\sim} Cz_1^{-\frac{3}{2}}e^{-\theta_1^+ z_1}.
\end{equation*}
\item
If $\rems{R\cdot \theta^+}=r \theta_1^+ -\mu_2=0$, then
$$
g(\theta_1)\underset{\theta_1^+}{\sim} \frac{1}{\sqrt{(\theta_1^+ -\theta_1)(\theta_1^+-\theta_1^-)}}\,,
$$
and so  $b=\theta_1^+$, $c_0=\frac{1}{\sqrt{(\theta_1^+-\theta_1^-)}}  $, $k=\frac{1}{2}$, $\Gamma(\frac{1}{2})=\sqrt{\pi}$. By the transfer theorems, 
\begin{equation*}
\nu (z_1) \underset{z_1\to +\infty}{\sim} Bz_1^{-\frac{-1}{2}}e^{-\theta_1^+ z_1}.
\end{equation*}
\item If $\rems{R\cdot \theta^+}=r \theta_1^+ -\mu_2> 0$, then
$$
g(\theta_1)\underset{\theta_1^p}{\sim} \frac{Res_{\theta_1^p}(g)}{\theta_1-\theta_1^p}\,,
$$
and so $b=\theta_1^p$, $c_0=Res_{\theta_1^p} (g)$, $k=1$, $\Gamma(1)=1$. By the transfer theorems, 
\begin{equation*}
\nu (z_1) \underset{z_1\to +\infty}{\sim} -Res_{\theta_1^p} (g) e^{-\theta_1^+ z_1}.
\end{equation*}
\end{enumerate}
\indent  \indent \indent We proceed to compute the residues to obtain explicit expressions for the constants. Let
\begin{equation*}
h(\theta_1)  := -r\theta_1 + \mu_2 + \sqrt{\left(\mu_1^2 + \mu_2^2\right) - (\theta_1 + \mu_1)^2}.
\end{equation*}
The first derivative of $h(\theta_1)$ is
\begin{equation*}
h'(\theta_1) = -r - \frac{\theta_1 + \mu_1}{\sqrt{\left(\mu_1^2 + \mu_2^2\right) - (\theta_1 + \mu_1)^2}}.
\end{equation*}
Since $\theta_1 =0$ and $\theta_1 = \theta_1^{p}$ are simple zeros of $h(\theta_1)$, we have that
\begin{equation}
\frac{1}{Res_{0}(g)} = {h'(0)} = -r - \frac{\mu_1}{|\mu_2|} = \frac{\mu_1 -r\mu_2}{\mu_2}.
\label{eq:res0}
\end{equation}
Then 
\begin{equation}
\frac{1}{Res_{\theta_1^p}(g)} = {h'\left(\theta_1^p\right)}
= (1+r^2) \,\frac{\mu_1 - r\mu_2}{(r^2-1)\mu_2 - 2r\mu_1},
\label{eq:resp}
\end{equation}
provided that $\theta_1^p$ is a zero of $h(\theta_1)$. 
Equations \eqref{eq:res0} and \eqref{eq:resp} give the values of $A$ and $D$, thereby completing the proof.
\end{proof}

\section{Inverse Laplace transform}
\label{sec:inverselaplace}
\noindent The transfer lemmas in the previous section only apply to univariate functions, and hence cannot be applied to the function $f$. In order to obtain the asymptotics of the occupancy density $\pii$, we first invert the two dimensional Laplace transform $f$. We then proceed to reduce its inverse 
to a single valued integral which gives an explicit expression of $\pii$. All of the above tasks are accomplished by Proposition~\ref{lem:inverselaplace} below.

\begin{prop}
For any $(z_1,z_2)\in \mathbb{R}\times \mathbb{R}_+$ and $\epsilon>0$ sufficiently small, 
the density occupancy measure can be written as
$$
\pii(z_1,z_2)=
\frac{-1}{i\pi} \int_{\epsilon-i\infty}^{\epsilon+i\infty}  \frac{ e^{-z_1\theta_1-z_2\Theta_2^+(\theta_1)}}{r\theta_1+\Theta_2^-(\theta_1)}    \mathrm{d}\theta_1
=\frac{1}{i\pi} \int_{\epsilon-i\infty}^{\epsilon+i\infty}  e^{-z_1\theta_1-z_2\Theta_2^+(\theta_1)}g(\theta_1)    \mathrm{d}\theta_1.
$$
\label{lem:inverselaplace}
\end{prop}
\begin{proof}
By Proposition~\ref{prop:functionaleq}, the Laplace transform $f(\theta_1,\theta_2)$ converges in the set $F$  which, for $\epsilon>0$ sufficiently small, contains $(\epsilon+i\mathbb{R})\times (i\mathbb{R})$. Then, Laplace transform inversion (\cite[Theorem 24.3 and 24.4]{doetsch_introduction_1974} and \cite{brychkov_multidimensional_1992})
gives
\begin{align*}
\pii(z_1,z_2) &=\frac{1}{(2\pi i)^2} \int_{\epsilon-i\infty}^{\epsilon+i\infty} \int_{-i\infty}^{i\infty} e^{-z_1\theta_1-z_2\theta_2} f(\theta_1,\theta_2) \mathrm{d}\theta_1 \mathrm{d}\theta_2.
\end{align*}
Recall from Section \ref{sec:functionaleq} the kernel 
\begin{equation*}
Q(\theta)=\frac{1}{2}(\theta_2-\Theta_2^+(\theta_1))(\theta_2-\Theta_2^-(\theta_1)).
\end{equation*}
Equations \eqref{eq:functional_eq} and \eqref{eq:value:g} yield that
$$
f(\theta_1,\theta_2)=\frac{-1-(R \cdot \theta) g(\theta_1)}{Q(\theta)} =
\frac{2}{(\theta_2-\Theta_2^+(\theta_1))(r\theta_1+\Theta_2^-(\theta_1))},
$$
and 
$$
\pii(z_1,z_2) =
\frac{1}{2\pi i} \int_{\epsilon-i\infty}^{\epsilon+i\infty}  \frac{2 e^{-z_1\theta_1}}{r\theta_1+\Theta_2^-(\theta_1)}   \left(  \frac{1}{2 \pi i} \int_{-i\infty}^{+i\infty} e^{-z_2\theta_2} 
\frac{1}{\theta_2-\Theta_2^+(\theta_1)}
\mathrm{d}\theta_2 \right) \mathrm{d}\theta_1.
$$
We now need show that
\begin{equation}
\frac{1}{2\pi i} \int_{-i\infty}^{+i\infty} 
\frac{e^{-z_2\theta_2} }{\theta_2-\Theta_2^+(\theta_1)}
\mathrm{d}\theta_2
= - e^{-z_2\Theta_2^+ (\theta_1)}  .
\label{eq:intcauchy}
\end{equation}
For some $A>0$, denote the half circle 
\begin{equation*}
\mathcal{C}_A = \{ \theta_2\in\mathbb{C} : |\theta_2|=A \text{ and } \Re \theta_2>0 \}. 
\end{equation*}
We now employ Cauchy's integral formula, integrating on the closed contour of Figure \ref{fig:contour}. Paying close attention to the direction of orientation, we obtain
$$
\frac{1}{2\pi i} \left( \int_{iA}^{-iA}+\int_{\mathcal{C}_A} \right)  
\frac{e^{-z_2\theta_2} }{\theta_2-\Theta_2^+(\theta_1)}
\mathrm{d}\theta_2
= e^{-z_2\Theta_2^+ (\theta_1)}.
$$
Note that since we have assumed throughout that $\mu_2< 0$, we have $\Re \Theta_2^+(\theta_1) >0$. It now remains to take the limit of the integrals when $A\to\infty$ and to show that the limit of $\int_{\mathcal{C}_A}$ is zero. Indeed, 
\begin{equation*}
\int_{\mathcal{C}_A} \frac{e^{-z_2\theta_2}}{\theta_2-\Theta_2^+(\theta_1)} 
\mathrm{d}\theta_2 = \int_{-\frac{\pi}{2}}^{\frac{\pi}{2}} \frac{e^{-z_2Ae^{it}}}{Ae^{it}-\Theta_2^+(\theta_1)}iAe^{it} \mathrm{d}t,
\end{equation*} 
which, by dominated convergence, converges to $0$ when $A\to\infty$. 
We thus obtain \eqref{eq:intcauchy}, completing the proof.

\begin{figure}[hbtp]
\centering
\includegraphics[scale=0.7]{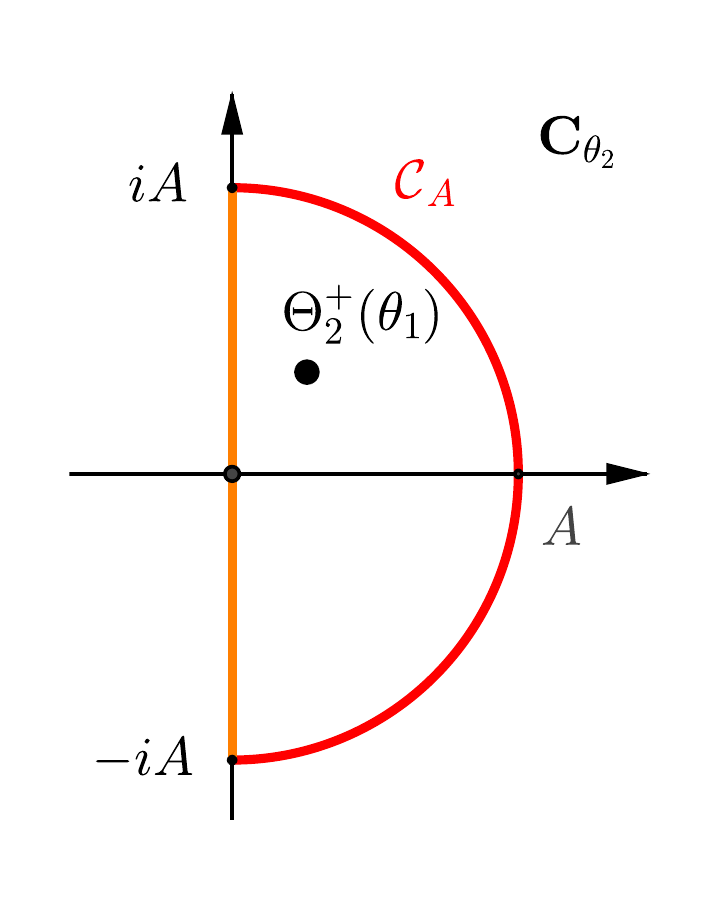}
\caption{Integration contour.}
\label{fig:contour}
\end{figure}

\end{proof}

\section{Saddle-point method and asymptotics} 
\label{sec:saddlepoint}
\noindent This goal of this section is to determine the exact asymptotic behavior of $\pii (\rho,\alpha)$ as $\rho\to \infty$ with $\alpha\in(0,\pi)$. Let $(\rho,\alpha)$ the polar coordinates of $z$, that is $\rho>0$, $\alpha\in(0,\pi)$ and $z=\rho e_\alpha$, where $e_\alpha=(\cos \alpha, \sin \alpha)$.
Let the saddle point be defined by
\begin{equation*}
\theta^\alpha :=(\theta^\alpha_1,\Theta_2^+(\theta^\alpha_1))=(-\mu_1+\cos\alpha \sqrt{\mu_1^2+\mu_2^2},-\mu_2+\sin\alpha \sqrt{\mu_1^2+\mu_2^2}),
\end{equation*}
and 
\begin{equation*}
\widetilde{\theta}^\alpha :=(\theta^\alpha_1,\Theta_2^-(\theta^\alpha_1))=(-\mu_1+\cos\alpha \sqrt{\mu_1^2+\mu_2^2},-\mu_2-\sin\alpha \sqrt{\mu_1^2+\mu_2^2}).
\end{equation*}
The poles are defined by
\begin{equation*}
\theta^p:=(\theta^p_1,\theta^p_2 )
=\left( \frac{2(r\mu_2-\mu_1) }{r^2+1},
 \Theta_2^+(\theta^p_1) \right),
\end{equation*}
and
$$
\theta^0 := (0, -2\mu_2).
$$
Recall that by Lemma~\ref{lem:MGFexplicit}, $\theta_1^0=0$ is a simple pole of $g(\theta_1)$. Further, if $\rems{R\cdot \theta^+}>0$, then $\theta_1^p$ is also a simple pole of $g$.
See Figure \ref{fig:cercle} below for a geometric interpretation of $\theta^\alpha$, $\theta^p$, $\theta^0$. We now proceed with the main theorem of the present paper.

\begin{figure}[hbtp]
\centering
\includegraphics[scale=0.2]{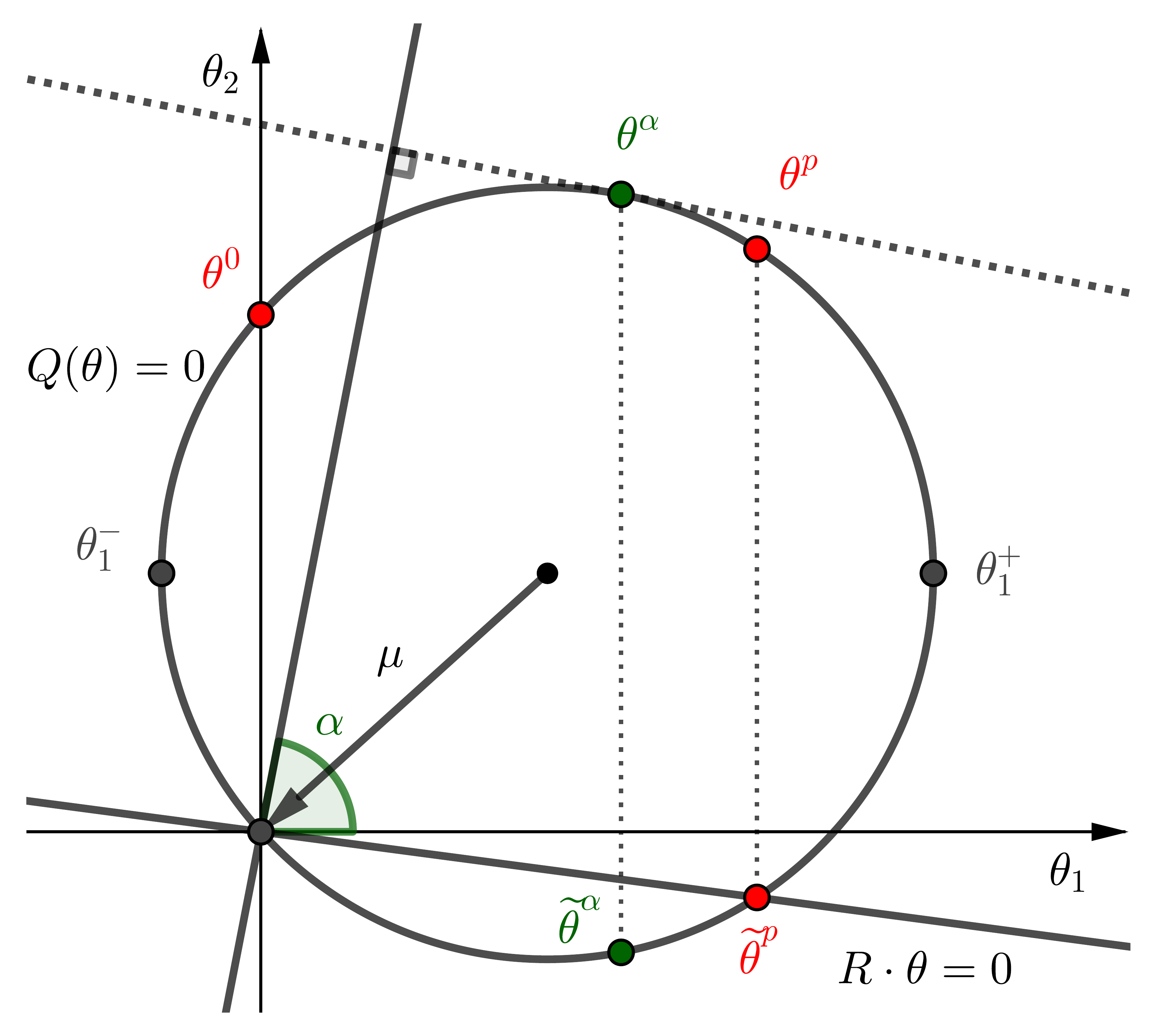}
\caption{The circle corresponds to $Q(\theta)=0$; the straight line corresponds to $R\cdot \theta =0$. The poles $\theta^p$ and $\theta^0$ are displayed in red and the saddle point $\theta^\alpha$ of $S$ is displayed in green.}
\label{fig:cercle}
\end{figure}

\begin{thm}
The asymptotic behavior of the occupancy density is given by
$$
\pii(\rho ,\alpha) \underset{\rho\to\infty}{\sim} 
\begin{cases}
 C_1{\rho^{-\frac{1}{2}}} e^{-\rho \theta^\alpha \cdot e_\alpha} & \text{if } 0<\theta^\alpha_1 < \theta^p_1 \text{ or } \rems{R\cdot \theta^+}\leqslant 0 ,
 \\
 C_2 e^{-\rho \theta^p \cdot e_\alpha} & \text{if } 0<\theta^p_1 
\leq \theta^\alpha_1  \text{ and } \rems{R\cdot \theta^+}>0 ,
 \\
  C_3 e^{-\rho \theta^0 \cdot e_\alpha} & \text{if } \theta^\alpha_1 
  \leq 0  ,
\end{cases}
$$
where 
$$
\theta^\alpha \cdot e_\alpha = -\mu \cdot e_\alpha + \Vert \mu \Vert,
\quad
\theta^p \cdot e_\alpha = \theta_1^p \cos \alpha +\theta_2^p\sin\alpha,
\quad
\theta^0 \cdot e_\alpha = -2\mu_2 \sin \alpha,
$$
and the constants satisfy
\begin{equation}
C_1= \sqrt{\frac{-2}{\pi  S''(\theta^\alpha_1)}} 
\frac{-1}{R\cdot \widetilde{\theta}^\alpha}
\quad
C_2= 2 (1+r^2) \,\frac{(r^2-1)\mu_2 - 2r\mu_1}{ r\mu_2-\mu_1 }, 
\quad
C_3=  \frac{2\mu_2}{\mu_1 -r\mu_2}
,
\label{eq:C2C3}
\end{equation}
when $\theta^\alpha_1\neq \theta^p_1  $ and $\theta^\alpha_1 \neq 0 $.
Furthermore, when a pole coincides with the saddle point, i.e. when $\theta^\alpha_1 =\theta^p_1 $ or $\theta^\alpha_1 = 0 $, the value of the constants $C_2$ and $C_3$ is half the value established in \eqref{eq:C2C3}. 
\label{thm:main}
\end{thm}

\begin{proof}
Let $S$ denote the function
$$
S(\theta_1)=\theta_1 \cos\alpha + \Theta_2^+(\theta_1) \sin \alpha.
$$
It is then straightforward to verify that
$$
\theta^\alpha_1=-\mu_1+\cos\alpha \sqrt{\mu_1^2+\mu_2^2}\,,
$$
is the saddle point of $S$, which means that $S'(\theta^\alpha_1)=0$ and $S''(\theta^\alpha_1)< 0$. 
By Proposition~\ref{lem:inverselaplace}, we have
$$
\pii(\rho ,\alpha)= \frac{-1}{i\pi} \int_{\epsilon-i\infty}^{\epsilon+i\infty}  \frac{ e^{-\rho S(\theta_1)}}{r\theta_1+\Theta_2^-(\theta_1)}    \mathrm{d}\theta_1
=\frac{1}{i\pi} \int_{\epsilon-i\infty}^{\epsilon+i\infty} e^{-\rho S(\theta_1)} g(\theta_1) 
\rems{ \mathrm{d}\theta_1}
.
$$
We now shift the contour of integration up to the saddle point (see Figure \ref{fig:contourcol} below). The curves of steepest descent are orthogonal. Let $\gamma_\alpha$ denote the steepest-descent contour near $\theta_1^\alpha$, that is $\Im S(\theta_1)=0$, which is orthogonal to the abscissa (for further details, see the orange curve on Figure~\ref{fig:contourcol} as well as the proof of Lemma~\ref{lem:saddlepoint} in the Appendix). We now proceed by analyzing two separate cases: $\theta_1^\alpha >0$ and $\theta_1^\alpha < 0$.
\begin{figure}[hbtp]
\centering
\includegraphics[scale=0.7]{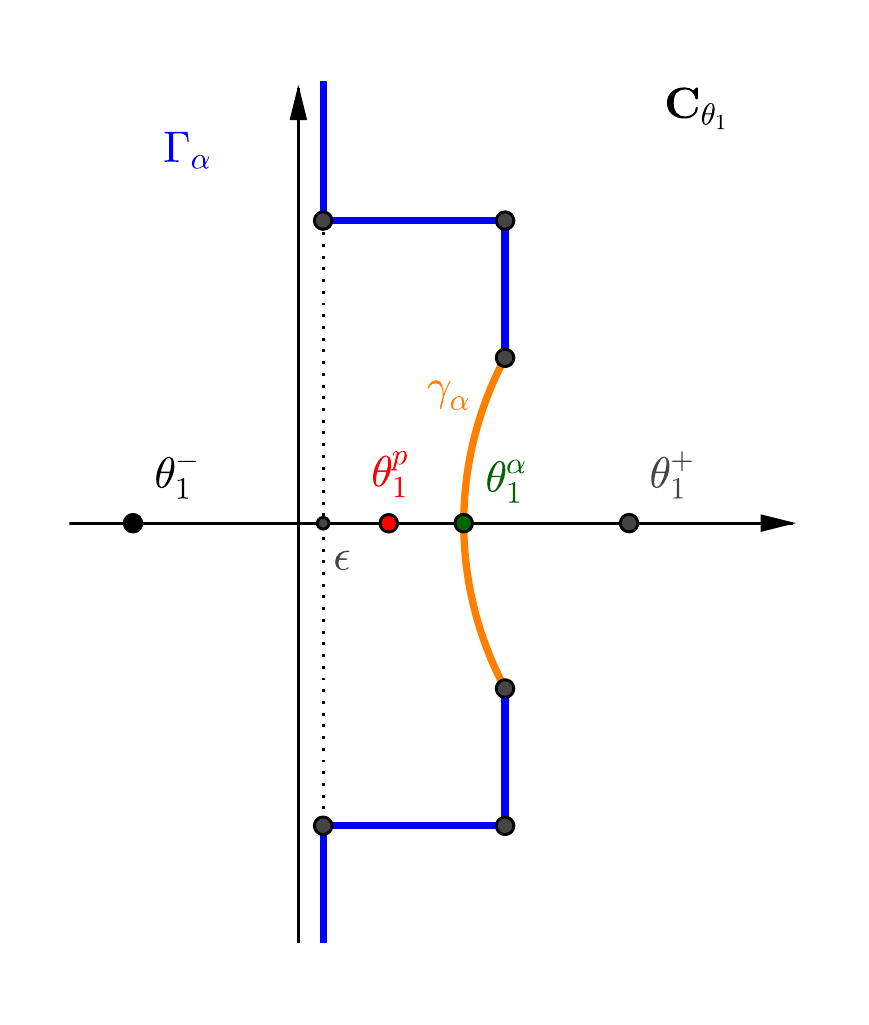}
\includegraphics[scale=0.7]{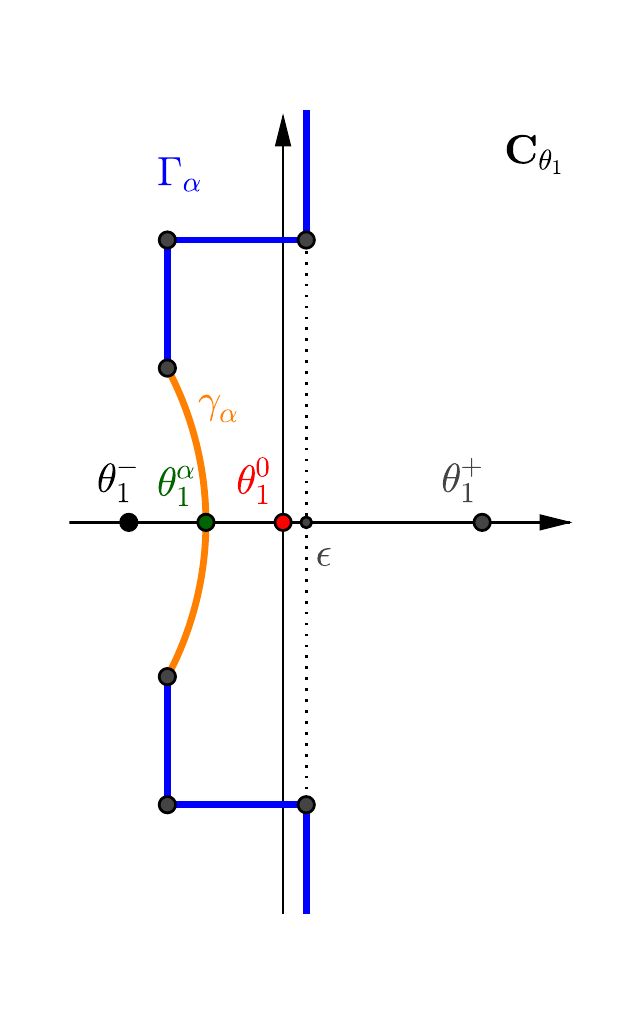}
\vspace{-0.7cm}
\caption{Shifting the contour. The left figure concerns the case $\theta_1^\alpha>0$. The right figure concerns the case $\theta_1^\alpha<0$.}
\label{fig:contourcol}
\end{figure}

\noindent \underline{Case I: $\theta_1^\alpha >0$.} Shifting the integration contour, it is possible to cross a simple pole $\theta_1^p$ coming from the zero $r\theta_1+\Theta_2^-(\theta_1)$, which itself is a pole of $g$. 
By Lemma~\ref{lem:MGFexplicit}, the function $g$ has a pole in $\theta^p_1$ if and only if $\rems{R\cdot \theta^+}>0$. Shifting the integration contour, a pole is then crossed if and only if $\theta_1^p<\theta_1^\alpha$ and $\rems{R\cdot \theta^+}>0$. Cauchy's formula gives
$$
\left( -\int_{\epsilon-i\infty}^{\epsilon+i\infty}+\int_{\gamma_\alpha}
+\int_{\Gamma_\alpha} \right) e^{-\rho S(\theta_1)} g(\theta_1) \mathrm{d}\theta_1=
\begin{cases}
0 & \text{if } 0< \theta^\alpha_1 < \theta^p_1 \text{ or }\rems{R\cdot \theta^+}\leqslant 0 ,
\\
2\pi i Res_{\theta_1^p} (g) e^{-\rho S(\theta_1^p)} & \text{if } 0<\theta^p_1 <\theta^\alpha_1  \text{ and } \rems{R\cdot \theta^+}>0  
 .
\end{cases}
$$
By the method of steepest descent (see \cite[\S 4 (1.53)]{fedoryuk_asymptotic_1989} as well as Lemma~\ref{lem:saddlepoint} in the Appendix),
$$
\int_{\gamma_\alpha}  e^{-\rho S(\theta_1)}g(\theta_1) \mathrm{d}\theta_1 
\underset{\rho\to\infty}{\sim} 
i \sqrt{\frac{-2\pi}{\rho S''(\theta^\alpha_1)}}  e^{-\rho \theta^\alpha \cdot e_\alpha}g(\theta_1^\alpha).
$$
Lemma~\ref{lem:intnegligeable} in the Appendix shows that the integral on the contour $\Gamma_\alpha$ is negligible compared to the integral on $\gamma_\alpha$. 
The asymptotics of $\pii$ are then given by the pole when $\theta^p_1 <\theta^\alpha_1$  (as $S(\theta_1^p)<S(\theta_1^\alpha)$), 
and by the saddle point otherwise.
We thus have that
$$
\pii(\rho ,\alpha) \underset{\rho\to\infty}{\sim} 
\begin{cases}
 \frac{C_1}{\sqrt{\rho}} e^{-\rho \theta^\alpha \cdot e_\alpha} & \text{if } \theta^\alpha_1 < \theta^p_1 \text{ or } \rems{R\cdot \theta^+}\leqslant 0 ,
 \\
 C_2 e^{-\rho \theta^p \cdot e_\alpha} & \text{if }\theta^p_1 < \theta^\alpha_1 \text{ and } \rems{R\cdot \theta^+}> 0 ,
\end{cases}
$$
where
\begin{equation*}
C_1= \sqrt{\frac{-2}{\pi  S''(\theta^\alpha_1)}} 
g(\theta_1^\alpha)
\quad \text{and} \quad
C_2= -2\text{Res}_{\theta_1^p} (g) = 2 (1+r^2) \,\frac{(r^2-1)\mu_2 - 2r\mu_1}{ r\mu_2-\mu_1 }  
.
\end{equation*}
The last equality above follows from \eqref{eq:resp}. Furthermore, from \eqref{eq:value:g} we have $g(\theta_1^\alpha)={-1} /{(R \cdot \widetilde{\theta}_\alpha)}$.
Lemma~\ref{polesaddle} of the Appendix deals with the final case in which $\theta_1^p=\theta_1^\alpha$. In this case the pole ``prevails'' and the asymptotics are given by $-\text{Res}_{\theta_1^p} (g)  e^{-\rho \theta^p \cdot e_\alpha} $.\\ 
 
\noindent \underline{Case II: $\theta_1^\alpha < 0$.} Shifting the integration contour, we cross the simple pole $\theta_1^0$ coming from the zero of $r\theta_1+\Theta_2^-(\theta_1)$, which itself is a pole of $g$.
Cauchy's formula then implies that, for $\theta^\alpha_1 < 0$,
$$
\left(\int_{\epsilon-i\infty}^{\epsilon+i\infty}+\int_{\gamma_\alpha}
+\int_{\Gamma_\alpha} \right) e^{-\rho S(\theta_1)} g(\theta_1) \mathrm{d}\theta_1=
2\pi i Res_0 (g) e^{-\rho S(0)}.
$$
The method of steepest descent (\cite[\S 4 (1.53)]{fedoryuk_asymptotic_1989}) yields
$$
\int_{\gamma_\alpha}  \frac{ e^{-\rho S(\theta_1)}}{r\theta_1+\Theta_2^-(\theta_1)}    \mathrm{d}\theta_1 
\underset{\rho\to\infty}{\sim} 
-i \sqrt{\frac{2\pi}{-\rho S''(\theta^\alpha_1)}}  \frac{e^{-\rho \theta^\alpha \cdot e_\alpha}}{(r\theta_1^\alpha+\Theta_2^-(\theta_1^\alpha)) }.
$$
Lemma~\ref{lem:intnegligeable} shows that the integral on the contour $\Gamma_\alpha$ is negligible in comparison to the integral on $\gamma_\alpha$. 
The asymptotics of $\pii$ are thus given by the pole since the contribution of the saddle point is negligible compared to that of the pole for $S(\theta_1^0)<S(\theta_1^\alpha)$. 
We thus have that
$$
\pii(\rho ,\alpha) \underset{\rho\to\infty}{\sim} 
  C_3 e^{-\rho \theta^0 \cdot e_\alpha}
   \text{ if }\theta^\alpha_1 < 0,
$$
where
\begin{equation*}
C_3=2\text{Res}_{0} (g)  = \frac{2\mu_2}{\mu_1 -r\mu_2}.
\end{equation*}
Note that the last equality above follows from  \eqref{eq:res0}.
The case in which $\theta_1^\alpha=0$ is relegated to Lemma~\ref{polesaddle} of the Appendix. In this final case, the asymptotics are given by $ \text{Res}_{0} (g) 
e^{-\rho \theta^0 \cdot e_\alpha}$. This concludes the proof {\textit{and closes Harrison's open problem}}.
\end{proof}


\begin{remark}
One can also use the saddle point method to determine asymptotics for all orders (see \cite[(1.22)]{fedoryuk_asymptotic_1989}). For all $n\in\mathbb{N}$, we have for some constants $c_k$ (where $c_0=C_1$) that
$$
\pii(\rho ,\alpha )=  C_2 e^{-\rho \theta^p \cdot e_\alpha}
\mathbf{1}_{\{ 0<\theta^p_1 \leq\theta^\alpha_1  \text{ and } r \theta_1^+ -\mu_2>0\}}
+
C_3 e^{-\rho \theta^0 \cdot e_\alpha}
\mathbf{1}_{\{ \theta^\alpha_1 \leqslant 0\}}
+
e^{-\rho \theta^\alpha \cdot e_\alpha} 
 \sum_{k=0}^n c_k {\rho^{-k-\frac{1}{2}}} 
 + o(e^{-\rho \theta^\alpha \cdot e_\alpha} \rho^{-n-\frac{1}{2}}).
$$

\end{remark}

\begin{remark}
The asymptotic behavior of the occupancy density for a non-reflected Brownian motion $B(t)+\mu t$ is given by $ \rho^{-\frac{1}{2}}e^{-\rho (\Vert \mu \Vert -\mu \cdot e_\alpha)}$.  Harrison explains this simpler case in his note \cite{Web}. Our results show that, 
when $0<\theta^\alpha_1 < \theta^p_1 \text{ or } \rems{R\cdot \theta^+}\leqslant 0$, the asymptotics are the same for both reflecting Brownian motion and for non-reflecting Brownian motion.
\end{remark}

Let $\alpha_\mu$ denote the angle between the $x$-axis and $-\mu$ (the opposite of the drift), and let $\alpha_R$ be the angle between the x-axis and $R$ (the reflection vector), as illustrated in Figure~\ref{fig:angles} below. 
\rems{We have $\alpha_R$ and $\alpha_\mu \in (0,\pi)$ and
$$
\tan \alpha_R = 1/r
\quad \text{and} \quad
\tan \alpha_\mu = \mu_1/\mu_2.
$$
Conditions \eqref{eq:conditiontrans} and \eqref{eq:conddriftneg} imply that $$\alpha_\mu < \alpha_R .$$}
As we have seen above, Theorem~\ref{thm:main} gives for a fixed angle $\alpha$ the asymptotic behavior of $\pii (\rho,\alpha)$ when $\rho\to\infty$ according to the value of the parameters $\mu$ and $R$. 
{It is also useful to state  the asymptotics  for fixed $\mu$ and $R$ and varying $\alpha$.} 
We do so in Corollary \ref{cor:asympt} below. See Figure~\ref{fig:angleasympt} for an illustration.

\begin{figure}[hbtp]
\centering
\includegraphics[trim = 0cm 1cm 0cm 1cm, scale=0.6]{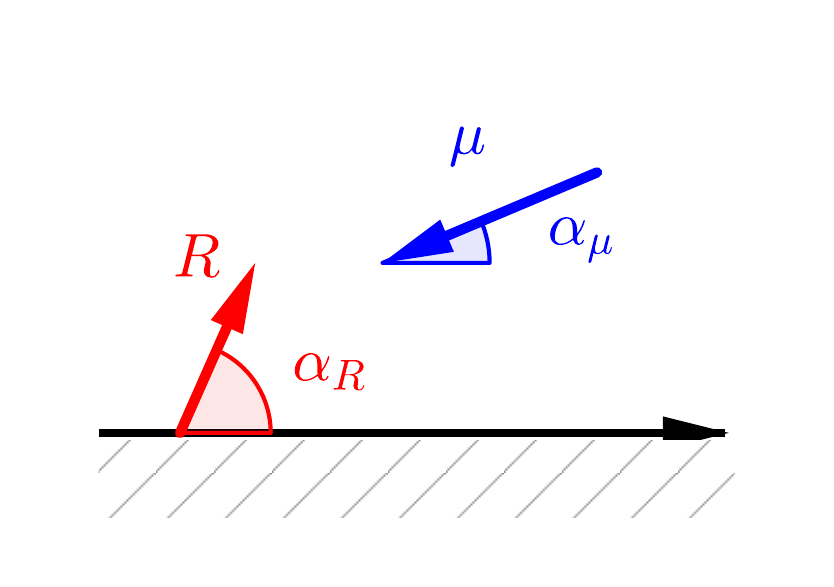}
\caption{Angles $\alpha_R$ and $\alpha_\mu$.}
\label{fig:angles}
\end{figure}

\begin{cor}
\label{cor:asympt}
Let us define
$$\alpha_0
:=\pi-\alpha_\mu \in(0,\pi)
\quad \text{and} \quad 
\alpha_1
:=\pi+\alpha_\mu-2\alpha_R\in (-\pi,\pi).$$ 
If \rems{$ \alpha_1\leqslant 0$}, then
$$
\pii(\rho ,\alpha) \underset{\rho\to\infty}{\sim} 
\begin{cases}
 C_1{\rho^{-\frac{1}{2}}} e^{-\rho \theta^\alpha \cdot e_\alpha} & \text{if } 0< \alpha < \alpha_0,
 \\
  C_3 e^{-\rho \theta^0 \cdot e_\alpha} & \text{if } \alpha_0 \leqslant \alpha <\pi  ,
\end{cases}
$$
and if \rems{$\alpha_1 > 0$}, then
$$
\pii(\rho ,\alpha) \underset{\rho\to\infty}{\sim} 
\begin{cases}
 C_1{\rho^{-\frac{1}{2}}} e^{-\rho \theta^\alpha \cdot e_\alpha} & \text{if }  \alpha_1 <\alpha < \alpha_0,
 \\
   C_2 e^{-\rho \theta^p \cdot e_\alpha} & \text{if }  0<\alpha \leqslant \alpha_1,
 \\
  C_3 e^{-\rho \theta^0 \cdot e_\alpha} & \text{if } \alpha_0 \leqslant \alpha <\pi .
\end{cases}
$$
\end{cor}
\begin{proof}
The proof follows immediately from Theorem~\ref{thm:main} by defining $\alpha_0$ and $\alpha_1$ such that both $\theta_1^{\alpha_1} = \theta_1^p$ and $\theta_1^{\alpha_0} = 0$. Doing so, we obtain
\begin{equation*}
\cos \alpha_0= \frac{\mu_1}{\sqrt{\mu_1^2+\mu_2^2}}=\cos(\pi-\alpha_\mu),
\end{equation*}
and
\begin{equation*}
\cos \alpha_1= \frac{\mu_1+\theta_1^p}{\sqrt{\mu_1^2+\mu_2^2}}=\cos(\pi+\alpha_\mu-2\alpha_R).
\end{equation*}
\rems{We remark that $R\cdot \theta^+>0$ is equivalent to $\alpha_1>0$. Straightforward calculation yields  
$$\tan \alpha_1 = \tan (\alpha_\mu -2\alpha_R) = \frac{\frac{\mu_2}{\mu_1}-\frac{2r}{r^2-1}}{1+\frac{\mu_2}{\mu_1}\frac{2r}{r^2-1}}.$$
Then $\tan \alpha_1=0$ is equivalent to $\frac{\mu_2}{\mu_1}-\frac{2r}{r^2-1}=0$. Under conditions \eqref{eq:conditiontrans} and \eqref{eq:conddriftneg}, this is equivalent to $R\cdot \theta^+=0$ . We conclude the proof by noting that $\pi+\alpha_\mu-2\alpha_R\in (-\pi,\pi)$.
}
\end{proof}
\begin{figure}[H]
\centering
\includegraphics[trim = 0cm 1cm 0cm 0cm,scale=0.8]{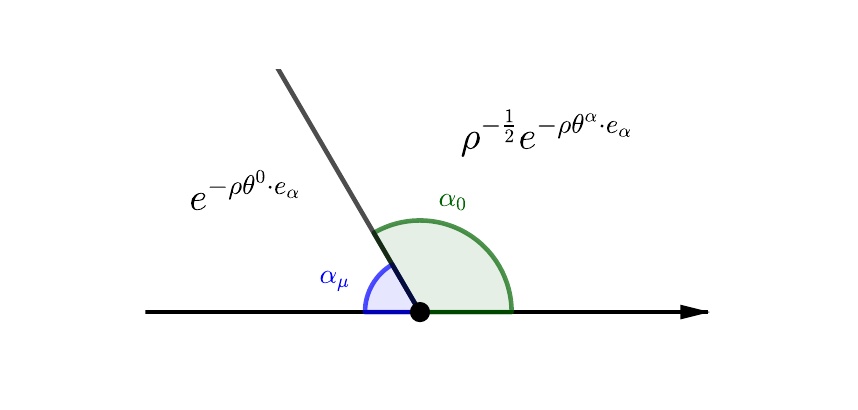}
\includegraphics[trim = 0cm 1cm 0cm 0cm,scale=0.8]{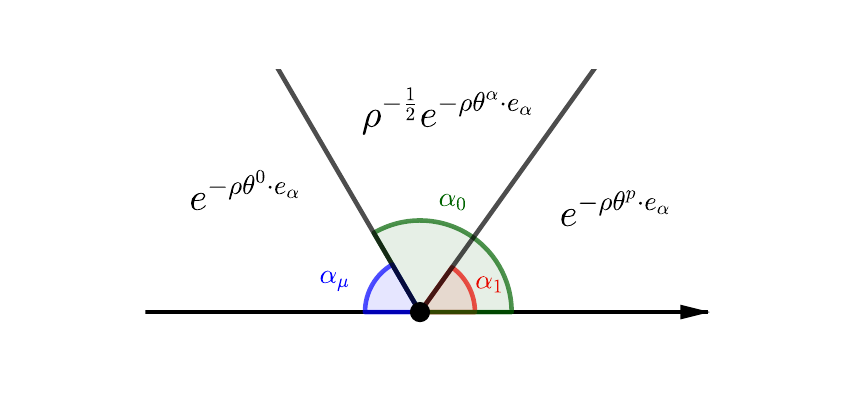}
\caption{Asymptotics by direction. The figure to the left considers the case \rems{$ \alpha_1 \leqslant 0$} and the figure to the right considers the case \rems{$ \alpha_1> 0$}.}
\label{fig:angleasympt}
\end{figure}

\section{Martin boundary}
\label{sec:martin}

\noindent {The goal of this section is to obtain the Martin boundary and the corresponding harmonic functions for the diffusion processes studied in this article. To this end, we recall the notion of harmonic function for a Markov process as well as the key relevant results from Martin boundary theory. We then proceed with the result in Proposition~\ref{prop:martin}.}

\indent Let $X(t)$ be a transient Markov process {on a state space $M$ (for example the upper half plane) and with transition density $p_t(x,y)$.} We recall a few definitions below. 

\begin{deff}
A function $h$ is harmonic in $M$ for the process $X$ (or $p_t$-harmonic) if the mean value property
$$\mathbb{E}_x \left[ h(X_{\tau_K}) \right] = h(x)$$
is satisfied for every compact $K\subset M$, where $\tau_K$ is the first exit time of $X$ from $K$. 
\end{deff}

\begin{deff}
The function $h$ is $p_t$-superharmonic if $\mathbb{E}_x \left[ f(X_{\tau_K}) \right] \leq f(x)$ for all compact $K$.
\end{deff}

\begin{deff} A non-negative harmonic function $h$ is minimal if for each harmonic function $g$ such that $0 \leq g \leq h$ we have $g = ch$ for some constant~$c$.
\end{deff}

\noindent The $\mathcal{C}^2$ harmonic functions for {$Z$, the reflected Brownian motion (RBM) in the upper half-plane}, are the functions which cancel the generator and the boundary generator, i.e. the functions $h\in \mathcal{C}^2 (\mathbb{R}\times\mathbb{R}_+)$ such that 
\begin{equation} \label{Martin1}
\mathcal{L}h =0,
\end{equation}
on the half plane and 
\begin{equation} \label{Martin2}
R\cdot \nabla h =0 
\end{equation}
on the abscissa. {This can be directly shown by the equality in \eqref{Ito}.} {Equations \eqref{Martin1} and \eqref{Martin2} imply} that a function is $p_t$-harmonic if it satisfies the classical Dirichlet problem in the half-plane with the oblique Neumann boundary condition.\\
\indent {We now recall a few relevant key results} in Martin boundary theory (for further details on Martin boundary theory, the reader may consult \cite{chung2006}, \cite{kunita1963,kunita1965}, and \cite{martin1941}). As in \eqref{def:greenfunction}, the Green's function is equal to 
\begin{equation*}
\pii^{x}(y):=\int_0^\infty p_t(x,y) \mathrm{d} t.
\end{equation*}
For some reference state $x_0$, the Martin kernel is defined {as}
$$
k^x_y:=\frac{\pii^{x}(y)}{\pii^{x_0}(y)}.
$$
The Martin compactification $\overline{M}$ is the smallest compactification of $M$ such that $y\mapsto k^x_y$ extends continuously. The Martin boundary is {defined as} the set
$$
\partial M := \overline{M} \setminus M.
$$
The function $x\mapsto k^x_y$ is superharmonic for all $y\in \overline{M}$. The {``minimal''} Martin boundary is defined~by
$$
\partial_m M := \{ y\in \partial M : 
x\mapsto k^x_y \text{ is minimal harmonic} \}.
$$
Finally, for any non-negative $p_t$-harmonic function $h$, there exists a unique finite measure $m$ such that for all $x\in M$,
$$
h(x)=\int_{\partial_m M} k^x_y m(\mathrm{d}y).
$$
{With these definitions and key results on Martin boundary theory in hand, we turn to Proposition \ref{prop:martin}.}
\begin{prop} Let $Z$ be the oblique RBM in the half plane starting from $x$ and let $k^x_y$ be its Martin kernel for the reference state $x_0=(0,0)$. Let us take $y=\rho e_\alpha$.
If $ \rems{\alpha_1}\leqslant 0$, then
$$
\lim_{\rho\to\infty} k^x_y =
\begin{cases}
 \left({(R\cdot {\theta}^\alpha)} e^{\widetilde{\theta}^\alpha\cdot x} - {(R\cdot \widetilde{\theta}^\alpha)} e^{{\theta}^\alpha\cdot x} \right)
\frac{ 1}{(\theta_2^\alpha- \widetilde{\theta}_2^\alpha)}  & \text{if } 0< \alpha < \alpha_0,
 \\
   1 & \text{if } \alpha_0 \leqslant \alpha <\pi  ,
\end{cases}
$$
and if $ \rems{\alpha_1} > 0$, then
$$
\lim_{\rho\to\infty} k^x_y =
\begin{cases}
  \left({(R\cdot {\theta}^\alpha)} e^{\widetilde{\theta}^\alpha\cdot x} - {(R\cdot \widetilde{\theta}^\alpha)} e^{{\theta}^\alpha\cdot x} \right)
\frac{ 1}{(\theta_2^\alpha- \widetilde{\theta}_2^\alpha)} & \text{if }  \alpha_1 <\alpha < \alpha_0,
 \\
 e^{\widetilde{\theta}^p \cdot x} & \text{if }  0<\alpha \leqslant \alpha_1,
 \\
   1 & \text{if } \alpha_0 \leqslant \alpha <\pi ,
\end{cases}
$$
where $\alpha_0$ and $\alpha_1$ are as defined in Corollary \ref{cor:asympt}. 
The Martin boundary coincides with the minimal Martin boundary and is homeomorphic to $[0,\alpha_0]$ if  $ \rems{\alpha_1}\leqslant 0$ and {is homeomorphic} to $[\alpha_1,\alpha_0]$ if  $ \rems{\alpha_1}> 0$. The {above} limits give all the harmonic functions of the minimal Martin boundary.
\label{prop:martin}
\end{prop}

\begin{proof}
To find the Martin boundary, it is sufficient to study the limits of the Martin kernel $k^x_y$ when $y\to\infty$ in each direction.
{Combining the results} in Corollary \ref{cor:asympt} and Appendix \ref{subsec:startingpoint} provides the asymptotics of $\pii^x(y)$, {that is,} the Green's function of the process starting from $x$. {It also implies the following two limits}.
{Firstly, if} $ \rems{\alpha_1}\leqslant 0$, then
$$
\lim_{\rho\to\infty} k^x_y =
\begin{cases}
 C_1(x) / C_1(0)
  & \text{if } 0< \alpha < \alpha_0,
 \\
   C_3(x) / C_3(0) & \text{if } \alpha_0 \leqslant \alpha <\pi.
\end{cases}
$$
{Secondly, if}  $ \rems{\alpha_1} > 0$, then
$$
\lim_{\rho\to\infty} k^x_y =
\begin{cases}
  C_1(x) / C_1(0) & \text{if }  \alpha_1 <\alpha < \alpha_0,
 \\
   C_2(x) / C_2(0) & \text{if }  0<\alpha \leqslant \alpha_1,
 \\
   C_3(x) / C_3(0) & \text{if } \alpha_0 \leqslant \alpha <\pi .
\end{cases}
$$
The constants $C_1(x)$, $C_2(x)$ and $C_3(x)$ are given by \eqref{eq:C1z0} and \eqref{eq:C23z0} in Appendix \ref{subsec:startingpoint}.
It is straightforward to verify that each of these functions are positive harmonic. They are also minimal. 
{We have thus provided all of the} harmonic functions of the
Martin compactification.
The Martin boundary coincides with the minimal Martin boundary and is homeomorphic to $[0,\alpha_0]$ if  $ \rems{\alpha_1}\leqslant 0$ and {is homeomorphic to} $[\alpha_1,\alpha_0]$ if  $ \rems{\alpha_1}> 0$.
\end{proof}

\begin{remark}
Proposition \ref{prop:martin} {gives} a similar result {to that} obtained in the discrete case for reflected random walks in the half plane \cite[Theorem 2.3]{kurkova_martin_1998}. 
The work of Ignatiouk-Robert \cite{ignatiouk-robert2010} states that the $t$-Martin boundary of a reflected random walk in a half-space is not stable. {It would be worthy to study this problem} in the case of reflected Brownian motion.
\end{remark}

\appendix
\section{Generalization of parameters}
\label{appendix:generalization}
\noindent The calculations in the main text were simplified by letting
$B(t)+\mu t $ be a two-dimensional Brownian motion with \textit{identity} covariance matrix and \textit{initial state} $(0,0)$. In Section~\ref{subsec:covmatric}, we show that the results of the present paper may be easily generalized to the case of a general covariance matrix $\Sigma$. In Section~\ref{subsec:startingpoint}, it is shown that our results may be generalized to the choice of any starting point $z_0$. As in the main text of the paper, we shall restrict our focus to $\mu_2< 0$; however, the same methodology shall apply to the case where $\mu_2\geq0$, as we shall show in  Section~\ref{subsec:driftpos}.

\subsection{Generalization to arbitrary covariance matrix}
\label{subsec:covmatric}

Let $\widetilde Z =(\widetilde Z_1,\widetilde Z_2) $ to be a reflected Brownian motion in the half-plane with  
covariance matrix
\begin{equation*}
\Sigma = \begin{pmatrix}
    \sigma_{11}      &  \sigma_{12}  \\ 
    \sigma_{12}  & \sigma_{22}
\end{pmatrix},
\end{equation*}
a drift $\widetilde \mu$, and a reflection vector $\widetilde{R} = (\widetilde{r},1).$ Let its occupancy density be denoted by $\widetilde \pii$. Consider the linear transformation given by $$T=\begin{pmatrix}
     \sqrt{\frac{\sigma_{11}\sigma_{22}}{\det \Sigma}}    &  0  \\ 
  \frac{-\sigma_{12}}{\sqrt{\sigma_{22}\det \Sigma}}   &  \frac{1}{\sqrt{\sigma_{22}}}
\end{pmatrix},$$ which satisfies $T \Sigma T^\top = Id $. Then $Z:=  \widetilde Z T\, $ is a reflected Brownian motion in the half-plane with identity covariance matrix, drift $\mu= \widetilde\mu T$, and reflection vector $R=\sqrt{\sigma_{22}} \widetilde R T=(r,1)$.
By a change of variables,  
we have that for all $\widetilde z\in\mathbb{R}\times\mathbb{R}_+$
\begin{equation}\label{equivalence}
\widetilde \pii (\widetilde z) =|\det T | \pii (\widetilde z T).
\end{equation} 
From equation \eqref{equivalence}, we may immediately derive the asymptotics of $\widetilde \pii$ from those of $\pii$.

\subsection{Initial state $x$}
\label{subsec:startingpoint}

In lieu of the initial state $(0,0)$, we now consider 
an arbitrary initial point $x=(x_1,x_2)\in\mathbb{R}\times \mathbb{R}_+$. We have 
$
Z(t)=x+B(t)+\mu t+ R \ell(t) 
$
where the local time of RBM on the abscissa is now
$$
\ell(t):= - \underset{0 \leqslant s \leqslant t}{\inf} {\{ 0\wedge (x_2+B_2 (s) +\mu_2 s)} \}.
$$
Recall Proposition \ref{prop:functionaleq}. The corresponding kernel functional equation to that of \eqref{eq:functional_eq} is
$$
0=e^{\theta\cdot x}+ Q(\theta)f(\theta) + (R \cdot \theta) g(\theta_1).
$$
The corresponding equation to that of \eqref{eq:value:g} is then
$$
g(\theta_1)
=\frac{-e^{(\theta_1,\Theta_2^- (\theta_1))\cdot x}}{r\theta_1+\Theta_2^- (\theta_1)}.
$$
Similarly to Proposition \ref{lem:inverselaplace},
we obtain 
$$
\pii^{x}(z_1,z_2)=
\frac{1}{i\pi} \int_{\epsilon-i\infty}^{\epsilon+i\infty} \left(e^{(\theta_1,\Theta_2^+ (\theta_1))\cdot x}-\frac{r\theta_1+\Theta_2^+ (\theta_1) }{r\theta_1+\Theta_2^-(\theta_1)} e^{(\theta_1,\Theta_2^- (\theta_1))\cdot x}\right)  \frac{ e^{-z_1\theta_1-z_2\Theta_2^+(\theta_1)}}{(\Theta_2^+(\theta_1)- \Theta_2^-(\theta_1))}    \mathrm{d}\theta_1
.
$$
Theorem \ref{thm:main} and Corollary \ref{cor:asympt} remain valid but with different constants depending of the starting point $x$. We obtain
\begin{equation}
C_1(x)= \sqrt{\frac{-2}{\pi  S''(\theta^\alpha_1)}} 
\left(e^{{\theta}^\alpha\cdot x}-\frac{R\cdot {\theta}^\alpha}{R\cdot \widetilde{\theta}^\alpha} e^{\widetilde{\theta}^\alpha\cdot x}\right)
\frac{ 1}{(\theta_2^\alpha- \widetilde{\theta}_2^\alpha)},
\label{eq:C1z0}
\end{equation}
and
\begin{equation}
C_2(x)= 2 (1+r^2) \,\frac{(r^2-1)\mu_2 - 2r\mu_1}{ r\mu_2-\mu_1 } e^{\widetilde{\theta}^p \cdot x}, 
\quad
C_3(x)=  \frac{2\mu_2}{\mu_1 -r\mu_2} 
,
\label{eq:C23z0}
\end{equation}
where 
$\widetilde{\theta}^p:=(\theta^p_1,
 \Theta_2^-(\theta^p_1) )$. Note that for $i\in\{1,2,3\}$ we have $C_i(0)=C_i$. 


\subsection{Case $\mu_2\geq 0$}
\label{subsec:driftpos}
We have assumed throughout that the inequality in \eqref{eq:conddriftneg} holds. 
We may use the exact same methodology we have developed for the case $\mu_2<0$ for the case $\mu_2>0$ or $\mu_2=0$. As the following results are obtained using straightforward calculations, the details are left to the reader. For $\mu_2>0$, we have the following:

\begin{enumerate}[label=(\roman*)]
\item The equality in \eqref{eq:value:g} remain valid and gives 
the value of the function $g$. 
However, $0$ is no longer a pole and the pole $\theta_1^p$ is negative if $\rems{R\cdot \theta^+}>0$. 
\item The asymptotics of $\nu_1$ are given by
\begin{equation*}
\nu_1 (z_1) 
\underset{z_1\to +\infty}{\sim}
\begin{cases}
A e^{-\theta_1^p z_1} & \text{if } \rems{R\cdot \theta^+}>0,
\\
B z_1^{-\frac{1}{2}} e^{-\theta_1^+ z_1} & \text{if } \rems{R\cdot \theta^+}=0,
\\
C z_1^{-\frac{3}{2}} e^{-\theta_1^+ z_1} & \text{if } \rems{R\cdot \theta^+}<0,
\end{cases}
\end{equation*}
and by 
$$
\nu_1 (z_1) \underset{z_1\to -\infty}{\sim} 
\begin{cases}
D e^{\theta_1^p z_1} & \text{if } \rems{R\cdot \theta^-}>0,
\\
E (-z_1)^{-\frac{1}{2}} e^{-\theta_1^- z_1} & \text{if } \rems{R\cdot \theta^-}=0,
\\
F (-z_1)^{-\frac{3}{2}} e^{-\theta_1^- z_1} & \text{if }\rems{R\cdot \theta^-}<0.
\end{cases}
$$
\rems{where $\rems{R\cdot \theta^- = r\theta_1^- - \mu_2}$ with $\theta^-:=(\theta_1^-,\Theta_2^-(\theta_1^-))=(\theta_1^-,-\mu_2)$.}
\item The asymptotics of $\pii$ are given by
$$
\pii(\rho ,\alpha) \underset{\rho\to\infty}{\sim} 
\begin{cases}
 C_1{\rho^{-\frac{1}{2}}} e^{-\rho \theta^\alpha \cdot e_\alpha} & \text{if } 0 \leqslant \theta^\alpha_1  < \theta^p_1 
 \text{ or }
 \theta^p_1< \theta^\alpha_1  \leqslant 0   
 \text{ or } (\rems{R\cdot \theta^+}\leqslant 0 \text{ and } \rems{R\cdot \theta^-}\leqslant 0),
 \\
 C_2 e^{-\rho \theta^p \cdot e_\alpha} & \text{otherwise}.
\end{cases}
$$
\end{enumerate}
Similar results hold for $\mu_2=0$.

\section{Technical lemmas} \label{AppendixB}
\begin{lem}
We have that
\begin{equation}
Z_{1} (t) \leq (\mu_{1} + r \mu_{2}^{-})t + B_{1}(t) + |r| \sup_{0 \leq s \leq t} |B_{2}(s)|.  \label{eq2}
\end{equation}
If $(\mu_{1} + r \mu_{2}^{-})<0$ is verified then we have $Z_1(t)\to -\infty$ for $t\to\infty$.
\label{lem:transient}
\end{lem}
\begin{proof}
By the definition of $\ell(t)$, 
\begin{eqnarray*}
\ell(t) = \sup_{0 \leq s \leq t} \left(-B_{2}(s) - \mu_2 s\right)
\leq \sup_{0 \leq s \leq t} \left(-B_{2}(s)\right) + \sup_{0 \leq s \leq t} \left( - \mu_2 s\right)
= \sup_{0 \leq s \leq t} \left(-B_{2}(s)\right) + \mu_{2}^{-} t,
\end{eqnarray*}
and
\begin{eqnarray*}
\ell(t) = \sup_{0 \leq s \leq t} \left(-B_{2}(s) - \mu_2 s\right)
\geq \inf_{0 \leq s \leq t} \left(-B_{2}(s)\right) + \sup_{0 \leq s \leq t} \left( - \mu_2 s\right)
= \inf_{0 \leq s \leq t} \left(-B_{2}(s)\right) + \mu_{2}^{-} t.
\end{eqnarray*}
Together with the definition of $Z_{1}(t)$, we have
\begin{equation}
Z_{1} (t) \leq \begin{cases}
(\mu_{1} + r\mu_{2}^{-})t + B_{1}(t) + r \sup_{0 \leq s \leq t} \left(-B_{2}(s)\right), & \quad \text{if $r \geq 0$}, \\
(\mu_{1} + r \mu_{2}^{-})t + B_{1}(t) + r \inf_{0 \leq s \leq t} \left(-B_{2}(s)\right), & \quad \text{if $r < 0$}.
\end{cases}  \label{eq1}
\end{equation}
The inequality in \eqref{eq2} now immediately follows from \eqref{eq1}. 
\end{proof}

\begin{lem}
The saddle point method gives
\begin{equation}\label{saddleequation}
\int_{\gamma_\alpha} e^{-\rho S(\theta_1)}g(\theta_1) \mathrm{d}\theta_1
\underset{\rho\to\infty}{\sim}
i \sqrt{\frac{-2\pi}{\rho S''(\theta^\alpha_1)}}  e^{-\rho S(\theta_1^\alpha)}g(\theta_1^\alpha).
\end{equation}
\label{lem:saddlepoint}
\end{lem}
\begin{proof}
The reader may consult \cite[\S 4 (1.53)]{fedoryuk_asymptotic_1989} for details about the saddle point method. We first offer a heuristic proof of the Lemma, which we then follow with a formal proof. The main contribution to the integral in \eqref{saddleequation} is in the saddle point $\theta_1^\alpha$. For some $\delta>0$, the curve $\gamma_\alpha$ can be replaced by its tangent $[\theta_1^\alpha-i\delta,\theta_1^\alpha+i\delta]$. The Taylor series of $S$ is
$$
S(\theta_1^\alpha +it)=S(\theta_1^\alpha)- \frac{S''(\theta_1^\alpha)}{2} t^2 +o(t^2).
$$ 
We may proceed to calculate
\begin{align*}
\int_{\gamma_\alpha} e^{-\rho S(\theta_1)}g(\theta_1) \mathrm{d}\theta_1
& \underset{\rho\to\infty}{\sim} g(\theta_1^\alpha) \int_{-i\delta}^{i\delta}  e^{-\rho S(\theta_1)} \mathrm{d}\theta_1,
\\
& \underset{\rho\to\infty}{\sim} g(\theta_1^\alpha)  e^{-\rho S(\theta_1^\alpha)} \int_{-\delta}^{\delta}  e^{\rho \frac{ S''(\theta_1^\alpha)}{2} t^2 } i \mathrm{d}t,
\\
&\underset{\rho\to\infty}{\sim} i g(\theta_1^\alpha) e^{-\rho S(\theta_1^\alpha)} \sqrt{\frac{-2}{S''(\theta_1^\alpha) \rho}} \underbrace{\int_{-\infty}^{\infty}  e^{-u^2} \mathrm{d}u }_{=\sqrt{\pi}}, 
\\
&\underset{\rho\to\infty}{\sim}
i \sqrt{\frac{-2\pi}{\rho S''(\theta^\alpha_1)}}  e^{-\rho \theta^\alpha \cdot e_\alpha}g(\theta_1^\alpha).
\end{align*}

We now offer a rigorous proof. For $\Im S(\theta_1)=0$, there are two level curves which are orthogonal and which intersect at the saddle point $\theta_1^\alpha$. These curves are the curves of ``steepest descent'' of $\Re S(\theta_1)$. One of them the abscissa, namely $[\theta_1^-,\theta_1^+]$. The other curve, which we call $\gamma_\alpha$, is orthogonal to the abscissa in $\theta_1^\alpha$. Let $\gamma(t):[-1,1] \to \gamma_\alpha$  be a parametrization of $\gamma_\alpha$ such that $\gamma(0)=\theta_1^\alpha$ and $\gamma'(0)=i$.\,
Noting that  $S'(\gamma(0))=S'(\theta_1^\alpha)=0$,
the Taylor series expansion of $S$ is 
$$S(\gamma(t))-S(\gamma(0))=\frac{t^2}{2}(\gamma'(0))^2 S''(\gamma(0))+o(t^2) = -\frac{t^2}{2}S''(\theta_1^\alpha) (1+o(1)). $$ 
Since $S''(\theta_1^\alpha)<0$, there exists a $\mathcal{C}^1$-diffeomorphic 
function $u$ defined in a neighborhood of $0$ such that
$$S(\gamma(t))-S(\theta_1^\alpha)=-\frac{t^2}{2}S''(\theta_1^\alpha)+o(t^2) = u^2(t).$$ 
The yields that
$$
u(t)=t\sqrt{\frac{-S''(\theta_1^\alpha)}{2}}+o(t).
$$
Note that $u(-1)<0$ and $u(1)>0$. Let the inverse of $u$ be $v=u^{-1}$. Then $v(0)=0$ and 
\begin{equation*}
v'(0)=\frac{1}{u'(0)}=\sqrt{\frac{-2}{S''(\theta_1^\alpha)}}.
\end{equation*}
We proceed to calculate
\begin{align*}
\int_{\gamma_\alpha} e^{-\rho S(\theta_1)}g(\theta_1) \mathrm{d}\theta_1
&= \int_{-1}^{1}  e^{-\rho S(\gamma(t))}g(\gamma(t)) \gamma'(t)\mathrm{d}t
\\
&= e^{-\rho S(\theta_1^\alpha)} \int_{-1}^{1}  e^{-\rho u^2(t)}g(\gamma(t)) \gamma'(t) \mathrm{d}t
\\
& \text{with a change of variables} \quad u(t)=s \text{ and } t=v(s)
\\
&= e^{-\rho S(\theta_1^\alpha)} \int_{u(-1)}^{u(1)}  e^{-\rho s^2}g(\gamma(v(s))) \gamma'(v(s)) v'(s) \mathrm{d}s
\\
&\underset{\rho\to\infty}{\sim}
 e^{-\rho S(\theta_1^\alpha)} 
  \underbrace{g(\gamma(v(0))) }_{=g(\theta_1^\alpha)}
  \underbrace{\gamma'(v(0)) }_{=i}
 \underbrace{v'(0)}_{=\sqrt{\frac{-2}{S''(\theta_1^\alpha)}}}
   \underbrace{\int_{u(-1)}^{u(1)}  e^{-\rho s^2} \mathrm{d}s}_{\sim\sqrt{\frac{\pi}{\rho}}}
\\
&\underset{\rho\to\infty}{\sim}
i \sqrt{\frac{-2\pi}{\rho S''(\theta^\alpha_1)}}  e^{-\rho S(\theta^\alpha)}g(\theta_1^\alpha).
\end{align*}
\end{proof}
\begin{lem}
For some $c_1>0$, $c_2>0$, and $B>0$, the following two statements hold:
\begin{enumerate}
\item \label{1} $\underset{a\in[\theta_1^+,\theta_1^-]}{\inf} \Re S(a+ib) \geqslant -c_1 + c_2 \sin\alpha |b|$ for all $|b| \geqslant B$,
\item \label{2} For fixed $a\in[\theta_1^+,\theta_1^-]$, the function $b\mapsto \Re S(a+ib)$ is increasing on $[0,\infty)$ and decreasing on $(-\infty,0]$.
\item \rems{For $\delta>0$ sufficiently small and for all $|b|\leqslant \delta$, we have,  for some $C_2>0$, $Re(S(ib)) \geqslant S(0) + C_2 |b|^2$.}
\end{enumerate}
\label{lem:technique}
\end{lem}
\begin{proof}
We first calculate 
\begin{equation*}
\Re S(a+ib)=a\cos\alpha -\mu_2\sin\alpha +\sin\alpha \Re \sqrt{\mu_1^2+\mu_2^2-(a+ib +\mu_1)^2}.
\end{equation*} 
The claimed properties then follow from straightforward calculus. For further details, we refer the reader to the proof of Lemma~19 of \cite{franceschi_asymptotic_2016}.
\end{proof}
\begin{lem}
We may choose $\Gamma_\alpha$ and $\gamma_\alpha$ such that
$$
\int_{\Gamma_\alpha} e^{-\rho S(\theta_1)}g(\theta_1) \mathrm{d}\theta_1 = o \left( \int_{\gamma_\alpha} e^{-\rho S(\theta_1)}g(\theta_1) \mathrm{d}\theta_1 \right).
$$
\label{lem:intnegligeable}
\end{lem}
\begin{proof}
Note that $\gamma_\alpha$ is the contour of steepest descent. Recall that the saddle point $\theta_1^\alpha$ is a minimum of $\Re S$ on the curve $\gamma_\alpha$ and note that $\Re S$ is increasing as one moves away from $\theta_1^\alpha$. For $\delta>0$, let $A\pm i\widetilde B$ be the endpoints of $ \gamma_\alpha$ chosen such that $S(A\pm i \widetilde B)= S(\theta_1^\alpha) +\delta$.
For $B$ sufficiently large, we shall choose a contour $\Gamma_\alpha$ such that (see Figure~\ref{fig:contourcollem})
\begin{equation}\label{integraleq}
\int_{\Gamma_\alpha} = 
\int_{\epsilon-i \infty}^{\epsilon-iB}+
\int_{\epsilon-iB}^{A-i B}+
\int_{A-i B}^{A-i \widetilde B}+
\int_{A+i \widetilde B}^{A+i B}+
\int_{A+i B}^{\epsilon+iB}+
\int_{\epsilon+iB}^{\epsilon+i \infty}.
\end{equation}
\begin{figure}[hbtp]
\centering
\includegraphics[scale=0.8]{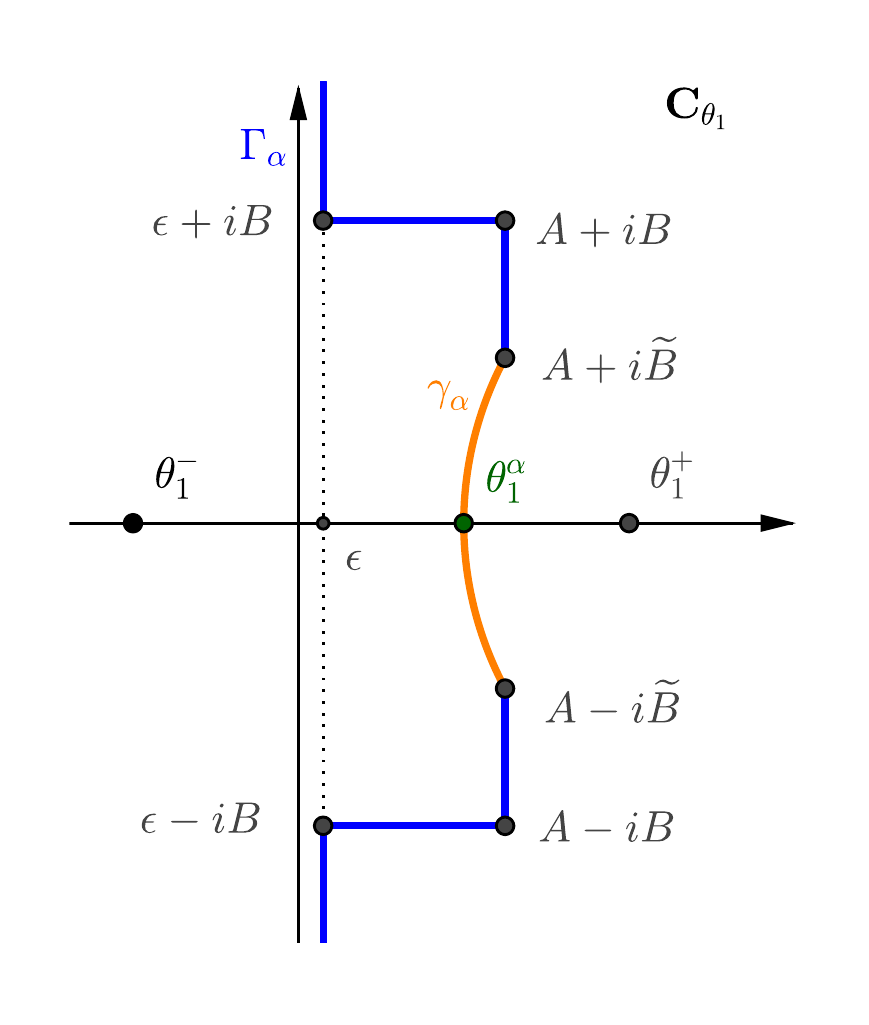}
\caption{The contour $\Gamma_\alpha$ and $\gamma_\alpha$}
\label{fig:contourcollem}
\end{figure}We now seek to show that, for some $\delta>0$, the six integrals in \eqref{integraleq} are $O(e^{-\rho (S(\theta_1^\alpha) +\delta)})$. Noting that $\Re S(\theta_1)=\Re S(\overline{\theta_1})$, it is enough to show this property for the last three integrals in \eqref{integraleq}. We first work with the third from the last integral of \eqref{integraleq}. By the \rems{second} statement in Lemma~\ref{lem:technique}, we have for all $\theta_1\in [{A+i \widetilde B},{A+i B}]$ that 
\begin{equation*}
\Re S(\theta_1)\geqslant S(A+i \widetilde B)= S(\theta_1^\alpha) +\delta. 
\end{equation*}
Then
$$
\left| \int_{A+i \widetilde B}^{A+i B}  e^{-\rho S(\theta_1)}g(\theta_1) \mathrm{d}\theta_1 \right|
\leqslant
e^{-\rho (S(\theta_1^\alpha) +\delta)}
\int_{A+i \widetilde B}^{A+i B} \left| g(\theta_1)\right| \mathrm{d}\theta_1.
$$
We continue with the second to last integral in \eqref{integraleq}. Let us consider $B$ such that 
\begin{equation*}
-c_1+c_2\sin\alpha B \geqslant S(\theta_1^\alpha) +\delta.
\end{equation*}
By the first statement in Lemma~\ref{lem:technique}, we have that for all $\theta_1\in [{A+i B},{\epsilon+i  B}]$,
\begin{equation*}
\Re S(\theta_1)\geqslant -c_1+c_2\sin\alpha B\geqslant S(\theta_1^\alpha) +\delta.
\end{equation*}
Thus
$$
\left| \int_{A+i B}^{\epsilon+iB} e^{-\rho S(\theta_1)}g(\theta_1) \mathrm{d}\theta_1 \right|
\leqslant
e^{-\rho (S(\theta_1^\alpha) +\delta)}
\int_{A+i B}^{\epsilon+iB} \left| g(\theta_1)\right| \mathrm{d}\theta_1.
$$
We now work with the final integral in \eqref{integraleq}. By the first statement of Lemma~\ref{lem:technique}, we have for all $\theta_1\in [0,\infty]$ that
\begin{equation*}
\Re S(\epsilon+iB+it)\geqslant -c_1+c_2\sin\alpha B +c_2t \geqslant S(\theta_1^\alpha) +\delta +c_2 t.
\end{equation*}
Then
$$\begin{aligned}
\left| \int_{\epsilon+iB}^{\epsilon+i \infty} e^{-\rho S(\theta_1)}g(\theta_1) \mathrm{d}\theta_1 \right|
& \leqslant 
\int_{0}^{ \infty} e^{-\rho \Re S(\epsilon+iB+it)} \left| g(\theta_1)\right| \mathrm{d}t
\\
& \leqslant 
e^{-\rho (S(\theta_1^\alpha) +\delta)}
\underbrace{
\int_{0}^{\infty} e^{-\rho c_2 t}\left| g(\epsilon+iB+it)\right| \mathrm{d}t.
}_{<\infty}
\end{aligned}$$
Combining the above results, we have
\begin{equation*}
\int_{\Gamma_\alpha} e^{-\rho S(\theta_1)}g(\theta_1) \mathrm{d}\theta_1=O(e^{-\rho (S(\theta_1^\alpha) +\delta)}).
\end{equation*}
The proof then concludes by applying Lemma~\ref{lem:saddlepoint}.
\end{proof}

\rems{The following Lemma considers the case where a pole coincides with the saddle point, a case left untreated in \cite{franceschi_asymptotic_2016}.}
\begin{lem}
If $\theta_1^\alpha=0$, then
\begin{align*}
\pii (\rho ,\alpha)
\underset{\rho\to\infty}{\sim} \text{Res}_{0}(g) e^{-\rho S(0)}.
\end{align*}
If $\theta_1^p=\theta_1^\alpha>0$, then
\begin{align*}
\pii (\rho ,\alpha)
\underset{\rho\to\infty}{\sim} -\text{Res}_{\theta_1^p}(g) e^{-\rho S(\theta_1^p)}  .
\end{align*}
\label{polesaddle}
\end{lem}

\begin{proof}
In these two cases the pole coincides with the saddle point. 
In this case we cannot integrate on the steepest descent contour because the integral will not converge. We thus (see Figure~\ref{fig:contourpol} below) employ alternative contours of integration near the pole. We shall consider two cases of interest separately. 
\begin{figure}[h]
\centering
\includegraphics[scale=0.7]{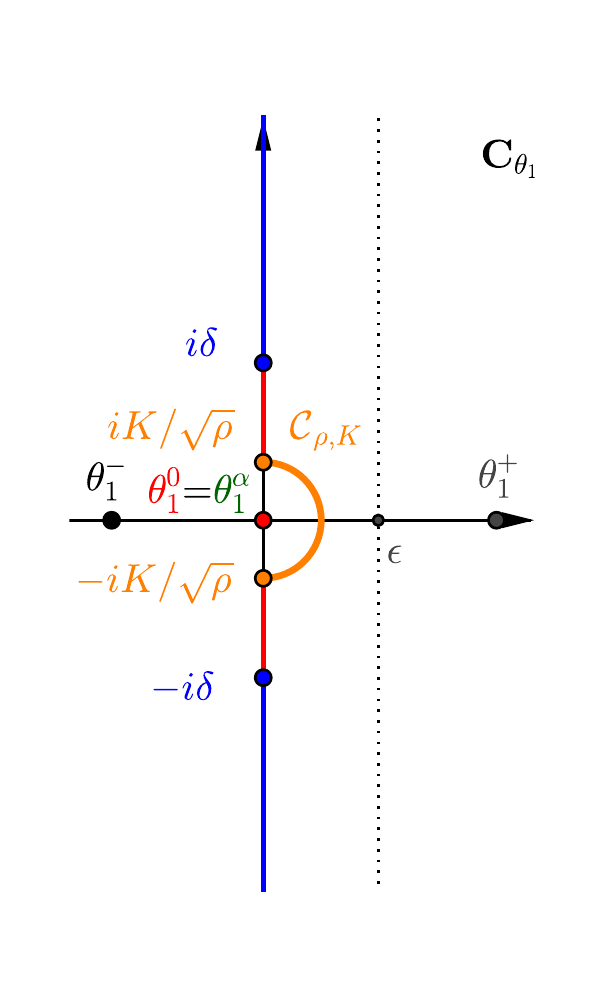}
\includegraphics[scale=0.7]{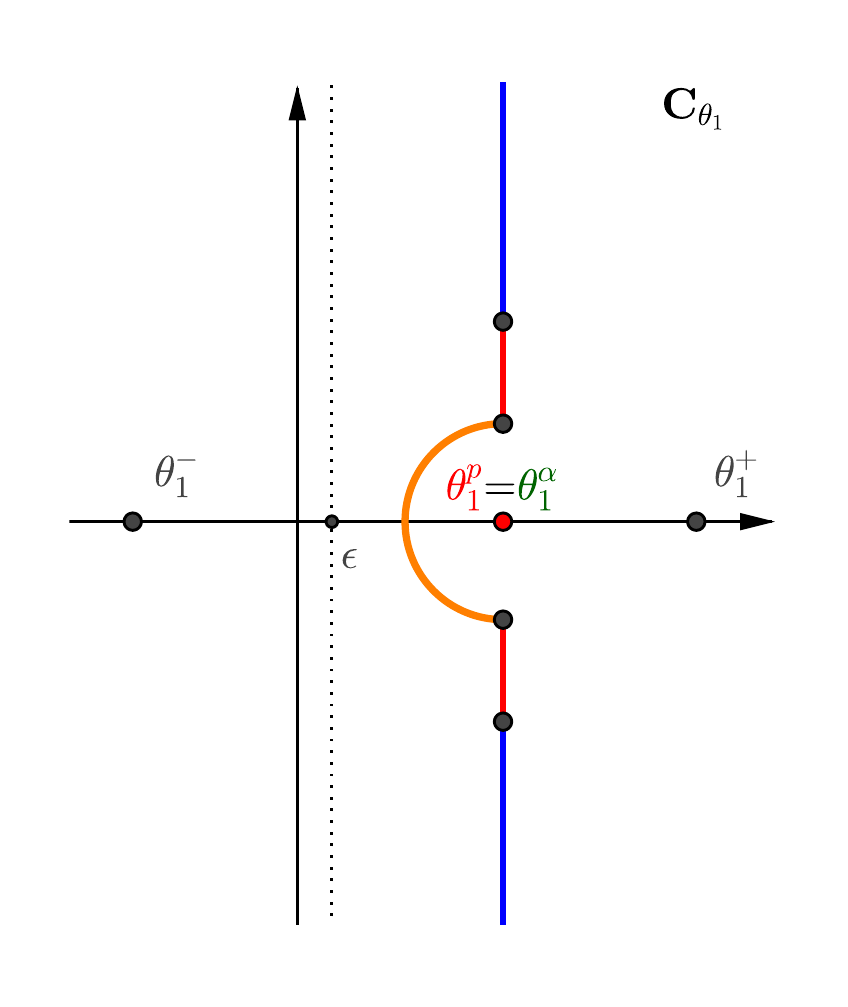}
\vspace{-0.7cm}
\caption{Shifting the contour. The left figure considers the case $\theta_1^\alpha=0$. The right figure considers the case $\theta_1^p=\theta_1^\alpha>0$.}
\label{fig:contourpol}
\end{figure}

\noindent \underline{Case I: $\theta_1^\alpha =0$.} \rems{For $K>0$}, consider the contour of integration  
\rems{\begin{equation*}
\mathcal{C}_{\rho, K} := \left\{ \theta_1\in\mathbb{C}: |\theta_1|= \frac{K}{\sqrt{\rho}} \text{ and } \Re \theta_1 \geqslant 0 \right\}
\end{equation*}}pictured in orange in Figure~\ref{fig:contourpol} below. The contour is half of \rems{a} small circle with center $0$ oriented in the positive direction. 
The Taylor series of $S$ is
\begin{equation*}
S(\theta_1)=S(0)+\frac{S'' (0)}{2} (\theta_1^2 +o(1))
\end{equation*}
In addition, 
\begin{equation*}
g(\theta_1)=\frac{\text{Res}_{0}(g)}{\theta_1} (1+o(1))
.
\end{equation*} We then have the following equivalence
\begin{align}
\int_{ \rems{\mathcal{C}_{\rho,K}}} e^{-\rho S(\theta_1)}g(\theta_1) \mathrm{d}\theta_1
& = \rems{ \text{Res}_{0}(g) e^{-\rho S(0)}   \int_{ \mathcal{C}_{\rho,K}}   e^{\rho \frac{ S''(0)}{2} \theta_1^2(1 +o(1)) } \frac{1+o(1)}{\theta_1} \mathrm{d}\theta_1 }
\\ 
&= \rems{ \text{Res}_{0}(g) e^{-\rho S(0)}   \int_{ \mathcal{C}_{1,K}}   e^{ \frac{ S''(0)}{2} t^2(1 +o(1/\sqrt{\rho})) } \frac{1+o(1/\sqrt{\rho})}{t} \mathrm{d}t }
\label{eq:equivint1}
\\ &\underset{\rho\to\infty}{\sim} \text{Res}_{0}(g) e^{-\rho S(0)}   \int_{\rems{\mathcal{C}_{1,K}}}  e^{\rho \frac{ S''(0)}{2} t^2 } \frac{1}{t}  \mathrm{d}t
=i\pi\,\text{Res}_{0}(g) e^{-\rho S(0)}.
\label{eq:equivint}
\end{align}
The equality in \eqref{eq:equivint1} comes from the change of variables \rems{$t=\sqrt{\rho}\theta_1$}. 
The last equality in \eqref{eq:equivint} comes from the fact that 
\begin{equation}\label{wholecircle}
\int_{\rems{\mathcal{C}_{1,K}}} \frac{1}{t}  e^{\rho \frac{ S''(0)}{2} t^2 } \mathrm{d}t= \frac{1}{2} \int_{\rems{|t|=K}} \frac{1}{t}  e^{\rho \frac{ S''(0)}{2} t^2 } \mathrm{d}t=i\pi, 
\end{equation}
where the equality in \eqref{wholecircle} illustrates that a change of variables enables us to integrate over the whole circle. 
Cauchy's residue theorem yields
\begin{equation}
\left( \int_{\epsilon-i\infty}^{\epsilon+i\infty}
-\int_{\rems{\mathcal{C}_{\rho,K}}}{-\int_{iK/\sqrt{\rho}}^{i\infty} -\int_{-i\infty}^{-iK/\sqrt{\rho}}}
 \right) e^{-\rho S(\theta_1)} g(\theta_1) \mathrm{d}\theta_1=
0.
\label{eq:contourdecompose}
\end{equation}
Using a similar argument in the proof of Lemma~\ref{lem:intnegligeable} yields that for $\delta>0$,
\begin{equation}
\left( \int_{i\delta}^{i\infty} +\int_{-i\infty}^{-i\delta} \right) e^{-\rho S(\theta_{1})}\, g(\theta_1)\, d\theta_{1} = o( e^{-\rho S(0)}).
\label{eq:negligeable}
\end{equation}
\rems{
For $\theta_1$ sufficiently small and for some $C_1>0$, we have $g(\theta_1)<C_1/|\theta_1|$. Invoking the third property of Lemma~\ref{lem:technique}, we have
\begin{eqnarray*}
&&\left|\int_{i K/\sqrt{\rho} }^{i \delta} e^{-\rho S(\theta_{1})}\, g(\theta_1)\, d\theta_{1} \right|
\leq 
e^{-\rho S(0)} \int_{ K/\sqrt{\rho} }^{ \delta} e^{- C_2 \rho t^2} \, \frac{C_1}{t} \,dt  
\leq  e^{-\rho S(0)} \int_{K}^{\infty} e^{-C_2 t^2} \, \frac{C_1}{t} \, dt .
\end{eqnarray*}
The same inequality holds for $\int_{-i K/\sqrt{\rho} }^{-i \delta}$.
Combining these two last inequalities with \eqref{eq:equivint}, \eqref{eq:contourdecompose} and \eqref{eq:negligeable} yields
\begin{equation*}
 \limsup_{\rho \rightarrow \infty} \left|  e^{\rho S(0)} \int_{\epsilon - i  \infty}^{\epsilon + i \infty} e^{-\rho S(\theta_{1})}\, g(\theta_1)\, d\theta_{1}  - i \pi \, \mathrm{Res}_{0}(g) \, e^{- \rho S(0)} \right| 
\leq \int_{K}^{\infty} e^{-C_2 l^2} \, \frac{C_1}{l} \, dl.
\end{equation*}
Letting $K \rightarrow \infty$,  and
recalling that by Proposition~\ref{lem:inverselaplace} we have}
\begin{equation*}
\pii(\rho ,\alpha)= \frac{1}{i\pi} \int_{\epsilon-i\infty}^{\epsilon+i\infty}  e^{-\rho S(\theta_1)} g(\theta_1) \mathrm{d}\theta_1,
\end{equation*}
the desired result follows.


\noindent \underline{Case II: $\theta_1^p=\theta_1^\alpha>0$.} The proof is identical to that of the previous case. The only difference is that we need to take into account that the orientation of the contour yields a minus sign. 

\end{proof}

\subsection*{Acknowledgments} 
We are deeply grateful to Professor J. Michael Harrison for sharing this problem with us as well as for providing some initial  ideas about it. We are also grateful \rems{to Irina Kourkova, Masakiyo Miyazawa,} and Kilian Raschel for helpful discussions about this problem. {We acknowledge, with thanks, Dongzhou Huang for helpful feedback. The first-named author thanks Rice University's Dobelman Family Junior Chair; he also gratefully acknowledges the support of ARO-YIP-71636-MA, NSF DMS-1811936, and ONR N00014-18-1-2192.}



\bibliographystyle{imsart-number} 


\bibliography{biblio2}


\end{document}